\documentclass[final,hidelinks,onefignum,onetabnum]{siamart220329}

\usepackage{amsfonts}
\usepackage{amssymb}
\usepackage{amsmath}
\usepackage{algorithmic}
\usepackage{graphicx}
\usepackage{subfigure}
\usepackage{todonotes}
\usepackage{algorithm}
\newtheorem{assumption}{Assumption}
\newtheorem{criterion}{Criterion}

\title{Adaptive stepsize algorithms for Langevin dynamics\thanks{A.L. acknowledges support of a MAC-MIGS CDT Scholarship (EPSRC grant EP/S023291/1) and the Luxembourg National Research Fund (FNR) (16978012). A.L., B.L., and J.L. would like to thank the Isaac Newton Institute for Mathematical Sciences for support and hospitality during the programme \textit{The mathematical and statistical foundation of future data-driven engineering} when work on this paper was undertaken (EPSRC grant EP/R014604/1).}}
\author{A. Leroy\thanks{School of Mathematics, The University of Edinburgh, UK. (\email{alix.leroy@sms.ed.ac.uk}, \email{b.leimkuhler@ed.ac.uk}, \email{d.j.higham@ed.ac.uk})} \and  B. Leimkuhler\footnotemark[2] 
\and J. Latz\thanks{Department of Mathematics, The University of Manchester, UK. (\email{jonas.latz@manchester.ac.uk})}
\and D. J. Higham\footnotemark[2]}

\def \rhoeq{\rho_{\rm eq}}
\def \rhoequn{{\rho}_{\rm eq}^{\rm under}}
\def \hatrhoequn{\hat{\rho}_{\rm eq}^{\rm under}}
\begin{document}
\makeatletter 
\@mparswitchfalse%
\makeatother
\normalmarginpar
\maketitle
\begin{abstract}
We discuss the design of an invariant measure-preserving transformation 
for the numerical treatment of Langevin dynamics based on a rescaling of time, with the goal of sampling from an invariant measure. Given an appropriate monitor function which characterizes the numerical difficulty of the problem as a function of the state of the system, this method allows stepsizes to be reduced only when necessary, facilitating efficient recovery of long-time behavior. We study both overdamped and underdamped Langevin dynamics. 
We investigate how an appropriate correction term that ensures preservation of the invariant measure should be incorporated into a numerical splitting scheme.
Finally, we demonstrate the use of the technique on several model systems, including a Bayesian sampling problem with a steep prior.
\end{abstract}



\section{Introduction}
\label{sec:background}

Underdamped Langevin dynamics describes 
random movements of a particle in contact with an infinite energy reservoir. 
The position vector $x \in \Omega \subset \mathbb{R}^d$ and the momentum vector $p=M\dot{x},\ p \in \mathbb{R}^d$, where $M$ is a $d\times d$ positive definite mass matrix,  describe the state of this particle throughout time. They satisfy the stochastic differential equations (SDEs):
\begin{align} \label{eq:full_langevin_1d}
\begin{split}
\begin{cases}
    d x  &= M^{-1} p d t, \\
    d p &= - \nabla_x V(x) d t - \gamma p d t + \sqrt{\frac{2 \gamma }{\beta}}M^{1/2} d W(t).
\end{cases}
\end{split}
\end{align}
Here $V: \mathbb{R}^d \rightarrow \mathbb{R}$ is a potential function, 
$(W(t))_{t \geq 0}$ is a standard Brownian motion on $\mathbb{R}^d$, $ \beta:= (k_{\rm B} T)^{-1}$ is a parameter representing the reciprocal temperature, where $k_{\rm B}$ denotes the Boltzmann's constant.
In Langevin dynamics a system loses energy through friction, controlled  by the parameter $\gamma$, and gains energy through stochastic fluctuations in such a way that the equilibrium state is described by the Gibbs-Boltzmann distribution with density $\rho_{\beta}(x,p) \propto \exp(-\beta H(x,p))$, where $H(x,p)=p^T{M^{-1}p/2+V(x)}$ is the energy of the system. 

The Langevin system and its overdamped counterpart (see Section \ref{sec:running examples})
have been widely used as sampling methods in many applications; notably in physics, chemistry, biology, social science, and machine learning (\cite{Donev2018,BeardSchlick2000,HutterOttinger1998, KantasParpasPavliotis2019,LeimkuhlerMatthews2015,Mao2015}). 
In sampling, the potential $V$ is chosen so that the SDE has a desired invariant measure, and hence samples consistent with this 
measure may be obtained from long time paths. This integration is most often performed using a numerical method with fixed stepsize. When the trajectories of the system undergo sudden changes, or exhibit highly oscillatory modes,  the stepsize must be small enough for the numerical method to remain stable. When the frequencies of the system vary as the system visits different regions, a fixed stepsize would need to be sufficiently small throughout the integration to reflect the most extreme scenario. These issues can lead to substantial computational overhead since the need for small stepsizes might be restricted to a region which is rarely visited. By varying the stepsize, adequate stability and accuracy can be maintained everywhere while computational costs are reduced. Automatic (variable) stepsize
selection has been incorporated into some existing molecular dynamics software packages, albeit in a ad hoc manner; for example, in Orient \cite{orient} and OpenMM \cite[\textit{VariableVerletIntegrator}]{openMM}. Situations that necessitate the use of an adaptive stepsize in Langevin dynamics include applications in Bayesian inference where the unknown parameters lie in bounded domains, e.g., in  
 machine learning when the parameters of an artificial neural network are constrained in order to improve
training \cite{Leimkuhler2021BetterTU}, or in Bayesian filtering \cite{VanTreesHarryL2007} -- we study problems of this form in Section \ref{sec:bayesian}. The propagation of rigid particles immersed in a Stokes fluid \cite{Donev2018} and ash cloud modelling \cite{KATSIOLIDES2018320} also require variable stepsizes for
the propagation of particles in boundary layers. These examples highlight the potential importance of variable stepsize discretization, using smaller steps only where needed in the simulation; hence allowing for both higher accuracy and faster convergence toward the target distribution. 

In this article, we introduce a time transformation for variable stepsize simulation of SDEs in a similar manner to that used in the deterministic setting \cite{HuangLeimkuhler1997} and meshing contexts \cite{Han2021,HuangRenRussel1994}. The transformation uses knowledge of a monitor function to modify the dynamics in both the overdamped and underdamped cases. In our context, a correction term arising from the fluctuation-dissipation theorem is required to ensure the desired target distribution. The focus in this article is on the calculation of averages with respect to the invariant probability measure defined by the Fokker-Planck equation associated to the SDE.
We first present an overview of related work, and introduce a simple example to illustrate the efficiency of the transformed dynamics in Section \ref{sec:running examples}. The main contributions of this work are as follows.
\begin{itemize}
    \item In Section \ref{sec:monitor_func}, a framework to design an efficient monitor function is presented, applicable even when little is known about the problem. We establish criteria on the function that are sufficient to maintain existence of the solution of the transformed dynamics. 
    \item In Section \ref{sec:correct_timestepping}, confirmation that the continuous overdamped transformed dynamics has the desired Gibbs distribution as a unique distribution when respecting the criteria outlined in Section \ref{sec:monitor_func}. 
    \item In Section \ref{sec:numerical_cons}, comparison of several numerical integrators to simulate the Gibbs distribution using the transformed underdamped dynamics.
    \item In Section \ref{sec:numerical_experiments} computational results for both overdamped and underdamped dynamics, in one and two dimensions, emphasizing the benefits and limitations of the method. 
\end{itemize}
We give brief conclusions in Section~\ref{sec:discussions_conclusions}.

There have been a number of approaches to designing variable stepsize strategies for strong (pathwise) approximation of SDEs \cite{Burrage2002,Gaines1997,Mauthner1998} with relatively little work addressing the issue of weak approximation \cite{Lemaire2007,Rossler2004,Szepessy2001,Valinejad2010,MerleProhl2021}. In particular, the finite time analysis of tamed numerical methods adopts a similar strategy: they control the size of the drift response by adapting the stepsize at each iteration \cite{Kelly2017,Neuenkirch2019,Sabanis2013}. Works focusing on the approximation of invariant measures are even rarer, since the relevant analyses, even in the case of fixed stepsize, are fairly recent. The authors of  \cite{Lamba2007} establish an adaptive algorithm preserving the ergodicity of the original SDEs and design an algorithm to control the drift term by halving or doubling the stepsize based on local error estimation and a user-provided tolerance. In \cite{MilsteinTretyakov2007}, the authors focus on demonstrating the ergodicity of quasi-symplectic integrators, i.e., integrators that, in the limit of small noise,  are symplectic, and they show that rejecting exploding trajectories does not affect ergodicity. Variable timestepping is certainly not the only approach to improve efficiency of Langevin simulations. For example, one may use an accept-reject step in conjunction with a Langevin proposal density \cite{RobertsTweedie1996bias} or Hamiltonian Monte Carlo \cite{betancourt2018conceptual,DUANE1987216,GirolamiCalderhead2011} to handle the bias that may arise. In principle such techniques could be combined with variable stepsize to further improve efficiency.

Preconditioning methods offer an alternative approach for improving the exploration of the space in the Hamiltonian case; however these tend to be difficult to implement and computationally expensive \cite{GirolamiCalderhead2011,MartinWilcox2012,RobertsStramer2002}. Related problems in sampling with the goal of improving convergence are discussed in \cite{Duncan_2016}; for example in heavy-tailed sampling, authors use a complementary strategy by inflating the stepsize in regions of slow changes \cite{heavytail}. Generally, the approach is to either increase the spectral gap or minimize the asymptotic variance of the approximation \cite{Lelievre2013, lelievre_stoltz_2016, RobertsTweedie1996}. Finally, similar methods have been exploited in the context of machine learning and stochastic optimization, as in \cite{MaChenFox2015,Mandt2016}, where adapting the learning rate is equivalent to varying the stepsize.

\section{Motivating example}
\label{sec:running examples}
We consider an example with a modified harmonic potential, where the force applied to the molecule depends on the frequency and has the form
\begin{align} \label{eq:pot_spring_der}
    F(x)=-V'(x) = (\omega(x)^2+c)x,\text{ with }
 \omega(x) = \frac{b}{\frac{b}{a}+(x-x_0)^2},
\end{align}
where $\omega: \mathbb{R} \rightarrow \mathbb{R}$ is the state dependent frequency. The potential $V$ is defined by 
\begin{multline} \label{eq:pot_spring}
        V(x) = \frac{1}{2} \biggl( a^{\frac{3}{2}}b^{\frac{1}{2}} x_0 \arctan \left (\frac{a}{b} (x-x_0)\right )  \\ 
         + \frac{a b (a(x-x_0)x_0-b)}{a (x-x_0)^2+b} + c (x-x_0)^2 + 2 c (x-x_0) x_0 \biggr).
\end{multline}
\begin{figure}[htp]
\hfill
\subfigure[The potential $V(x)$ and the function $\omega(x)$ with $a=5$, $x_0=0.5$ and $b=1$. A smaller value of $b$ increases the values of $\omega$ around $x_0$.]{\includegraphics[width=5cm]{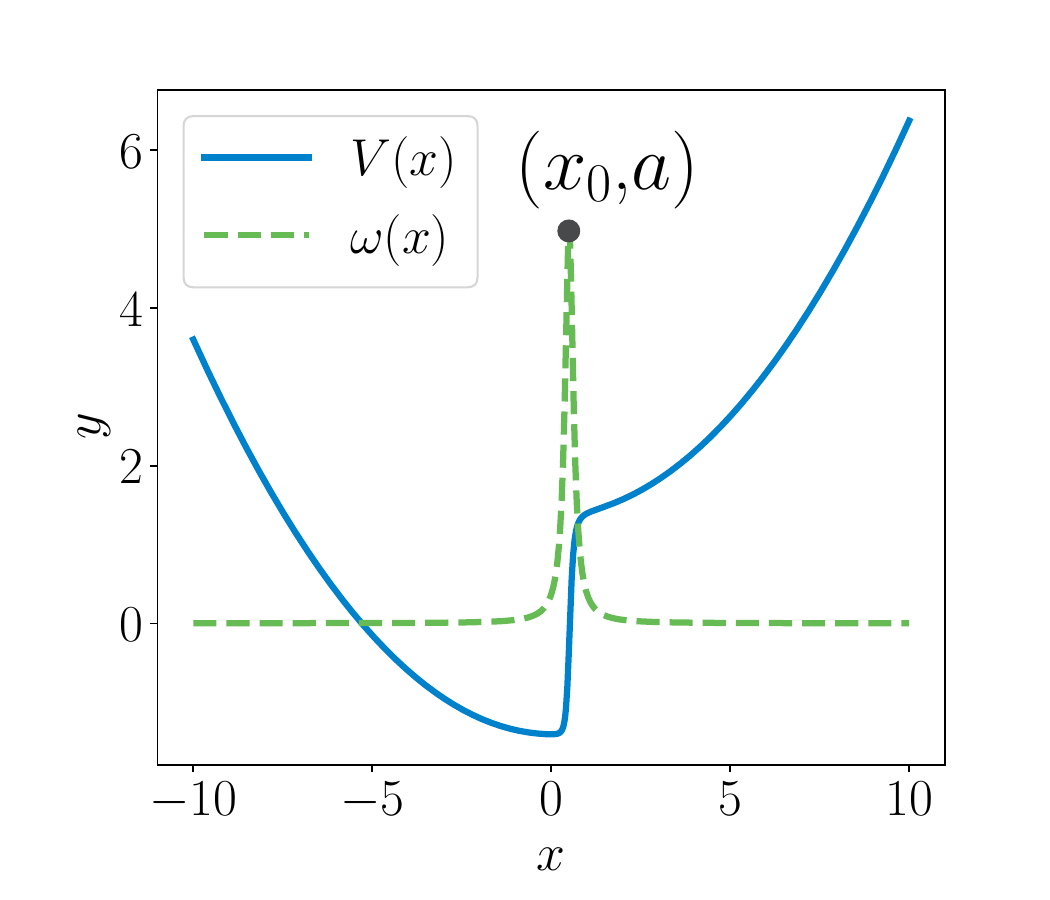}\label{fig:spring_left}}
\hfill
\subfigure[The gradient $V'$ and an example of monitor function $g$.]{\includegraphics[width=5cm]{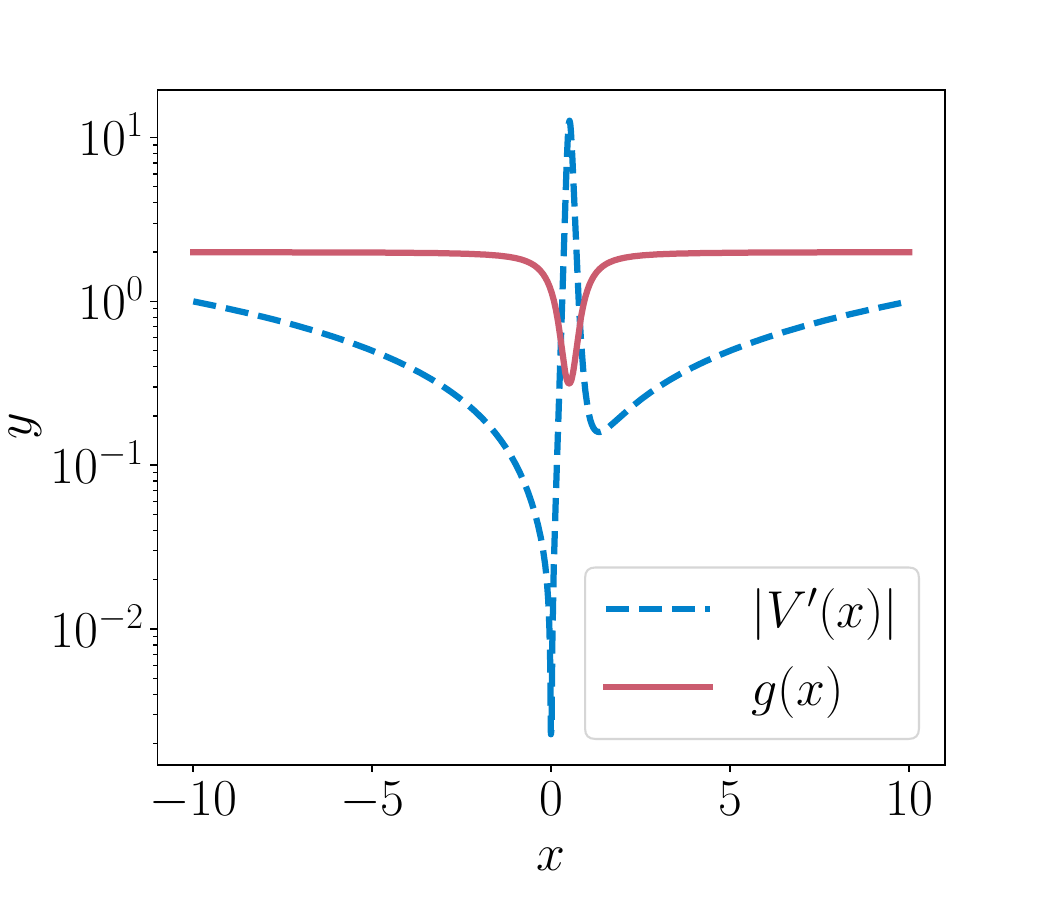}\label{fig:spring_right}}
\hfill
\caption{Plots of the potential, the absolute value of its derivative and an example of a monitor function $g$.}
\end{figure}
The Boltzmann-Gibbs distribution with density
\begin{align} \label{eq:Gibbs_eq}
\rho(x)=\frac{1}{Z} \exp(-\beta V(x)),
\end{align}
and normalising factor 
\begin{align} \label{eq:Gibbs_eq_Z}
Z = \int_{\mathbb{R}^d} \exp(-\beta V(x)) d x,
\end{align}
is the unique invariant distribution of the overdamped SDE 
\begin{align} \label{eq:original_sde}
    d x = - \nabla V(x) d t + \sqrt{2 \beta^{-1}} d W
\end{align}
by \cite[Proposition 4.2]{Pavliotis2014}. This result holds if the potential $V$ is smooth and confining.
The overdamped SDE \eqref{eq:original_sde} can be obtained from the underdamped system of SDEs \eqref{eq:full_langevin_1d} by asymptotic methods, letting the mass go to zero or the friction to infinity (see \cite[Sect. 6.5.1]{Pavliotis2014}). 
To obtain samples from the stationary distribution \eqref{eq:Gibbs_eq}, the standard approach is to discretize the overdamped SDE \eqref{eq:original_sde}  using the Euler-Maruyama scheme
\begin{align}\label{eq:em_sde}
    X_{n+1} =X_n - \nabla V(X_n) h + \sqrt{2 \beta^{-1} h} Z_n.
\end{align}
The stepsize $h$ is fixed, $X_{n} \approx X(nh)$ and the $Z_0, Z_1, \ldots \sim \mathcal{N}(0,1)$ form a sequence of i.i.d. standard normal random variable distributed.
Figure \ref{fig:spring_left} shows the values taken by the potential $V$, close to a harmonic potential in most of the domain but steepening near $x_0$. An example of an associated monitor function $g$ is provided in Figure \ref{fig:spring_right}. This follows the variation of $\frac{1}{|V'(x)|}$ whilst being bounded by strictly positive values $m<M$, effectively decreasing the stepsize around the singularities, see equation \eqref{eq:monitor_phi} in Section \ref{sec:monitor_func}. The parameters $m$ and $M$ define lower and upper bounds on the multiplicative factor of the stepsize.

In the next section, we describe a direct time-rescaled SDE which does not conserve the invariant measure, and then concentrate on invariant measure-preserving transformed SDE. Variable stepsize  integration is implemented by defining a new time variable $\tau$ such that $\tilde{x}(\tau)=x(t(\tau))$. For a given monitor function $g(x) >0$, for all $x \in \mathbb{R}^d$, the time $t$ is monotone in $\tau$ and defined by the relation
\begin{align} \label{eq:rescaling_dt}
  d t &= g(x(t(\tau))) d \tau, \\
  d W(t) &= \sqrt{g(\tilde{x})}d \widetilde{W}(\tau). 
\end{align}
Replacing the terms \eqref{eq:rescaling_dt} in the overdamped SDE \eqref{eq:original_sde}, the direct time-rescaled SDE is
\begin{equation}\label{eq:naive_time_rescaled}
d \tilde{x} =  -\nabla V(\tilde{x}) g(\tilde{x}) d \tau + \sqrt{2 \beta^{-1} g(\tilde{x})} d \widetilde{W}(\tau),
\end{equation} 
with the Euler-Maruyama numerical scheme yielding
\begin{equation} \label{eq:em_sde_naive}
    \tilde{X}_{n+1} =\tilde{X}_n - \nabla V(\tilde{X}_n) g(\tilde{X}_n) h + \sqrt{2 \beta^{-1} g(\tilde{X}_n)h} Z_n. 
\end{equation}
Using \eqref{eq:em_sde} and \eqref{eq:em_sde_naive}, we obtain samples using the discretized scheme and plot the histograms of the positions alongside the normaliced  stationary density \eqref{eq:Gibbs_eq} in Figure \ref{fig:histograms}.
\begin{figure}[htp]
\centering
\includegraphics[width=0.5\textwidth]{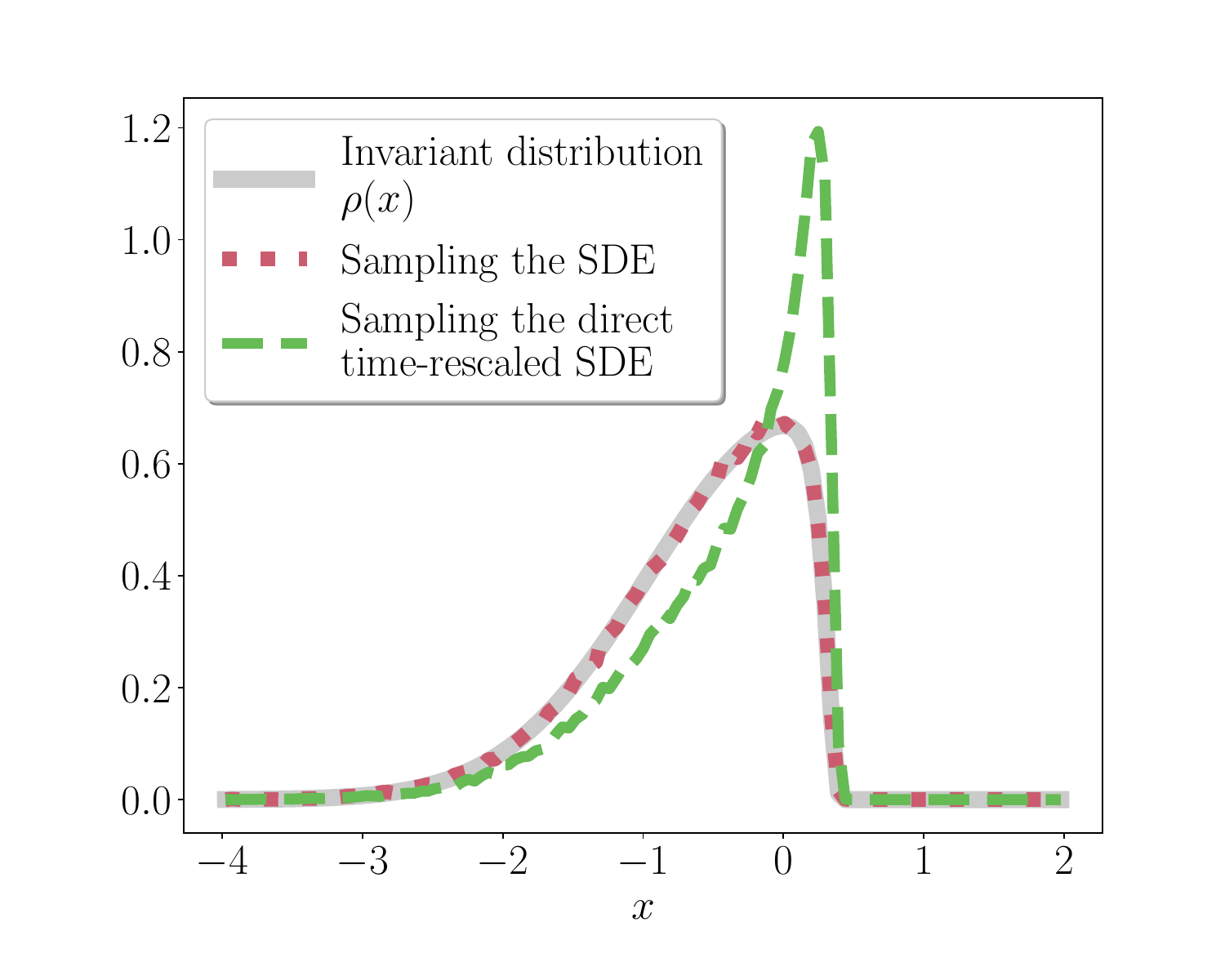}
    \caption{Histograms of the invariant distribution and samples obtained using the Euler-Maruyama scheme applied (a) to the original SDE \eqref{eq:original_sde} and (b) to the direct time-rescaled SDE \eqref{eq:naive_time_rescaled} after $50000$ steps, with step of size  $h=0.001$, temperature $\beta^{-1}=0.1$, and $10^5$ samples. The function $g$ is bounded by $M=2$, $m=0.001$. The parameters of the potential are set to $c=0.1$, $b=0.1$, $a=10$ and $x_0=0.5$.}
    \label{fig:histograms}
\end{figure}
The numerical scheme applied to the direct time-rescaled SDE in Equation \eqref{eq:em_sde_naive} does not yield samples that agree with the invariant distribution \eqref{eq:Gibbs_eq}. In fact this has nothing to do with the discretization scheme itself; it is a consequence only of the time-rescaling of the dynamics.  It can be understood by examining the adjoint operator of the infinitesimal generator associated with the direct time-rescaled equation \eqref{eq:naive_time_rescaled}:
\begin{equation} \label{eq:forw_kolm_naive}
        \mathcal{\tilde{L}}^* \rho(t,\tilde{x}) =- \nabla \cdot \left( -\nabla V(\tilde{x})g(\tilde{x}) \rho(t,\tilde{x}) + \beta^{-1} \nabla  g(\tilde{x})  \rho(t,\tilde{x}) \right).
\end{equation}
The invariant distribution \eqref{eq:Gibbs_eq} is a stationary density if it is in the kernel of the adjoint operator $\mathcal{\tilde{L}}^*$ \cite[Sect. 4.1]{Pavliotis2014}. We see that
\begin{equation} \label{eq:drift_term_naive}
        \mathcal{\tilde{L}}^* \rho(\tilde{x}) =- \beta \nabla V(\tilde{x}) \rho(\tilde{x}), 
\end{equation}
noting that the invariant distribution \eqref{eq:Gibbs_eq} respects $\nabla \rho(x) = -\beta \nabla V(x) \rho(x)$ by the chain rule. As $\mathcal{\tilde{L}}^* \rho(\tilde{x}) \neq 0$, the probability distribution \eqref{eq:Gibbs_eq} is not in the kernel of the adjoint operator $\mathcal{\tilde{L}}^*$.

This motivates a correction that leads to a invariant measure-preserving transformed SDE that retains the stationary distribution of the original SDE
\begin{align} \label{eq:transformed_overdamped_sde}
d x & = -g(x)\nabla V(x) d t  + \beta^{-1}  \nabla g(x) d t + \sqrt{2\beta^{-1} g(x)} d W.
\end{align}
From this point forward, we will only consider this invariant measure-preserving transformed SDE using $x$ and $t$ variables, which we refer to as the IP-transformed SDE. The associated Euler-Maruyama scheme is
\begin{equation} \label{eq:em_tr_sde}
    X_{n+1} =X_n - \nabla V(X_n) g(X_n) h + \beta^{-1}  \nabla g(X_n) h +\sqrt{2 \beta^{-1} g(X_n)h} Z_n. 
\end{equation}
We note that while the IP-transformed dynamics offers the advantage of using ``easy to implement'' schemes of higher weak order, we chose to use the Euler-Maruyama scheme as the aim is to compare the dynamics under similar discretization.

\begin{figure}[htp]
\hfill
\subfigure[Samples obtained with step of size $h=0.001$ to reach the final time $T_f=50$.]{\includegraphics[width=5cm]{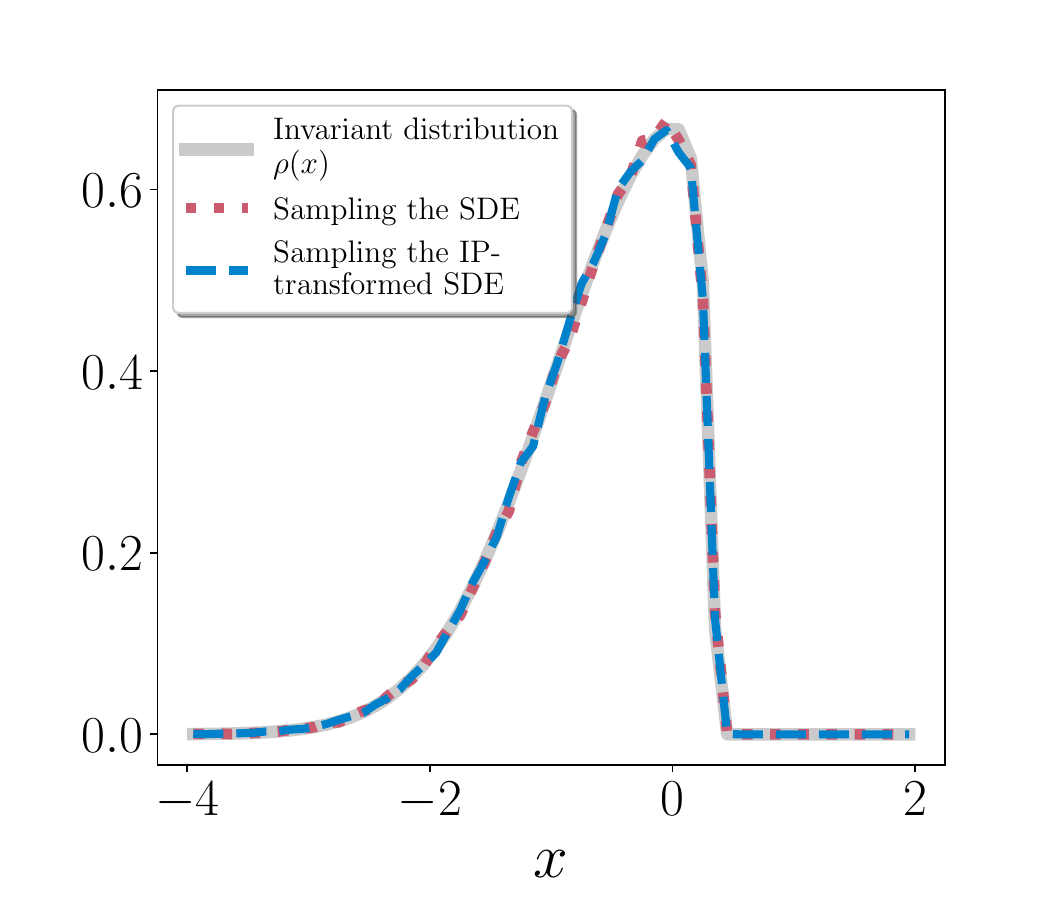}\label{fig:tr_small_h}}
\hfill
\subfigure[Samples obtained with step of size $h=0.05$ to reach the final time $T_f=70$.]{\includegraphics[width=5cm]{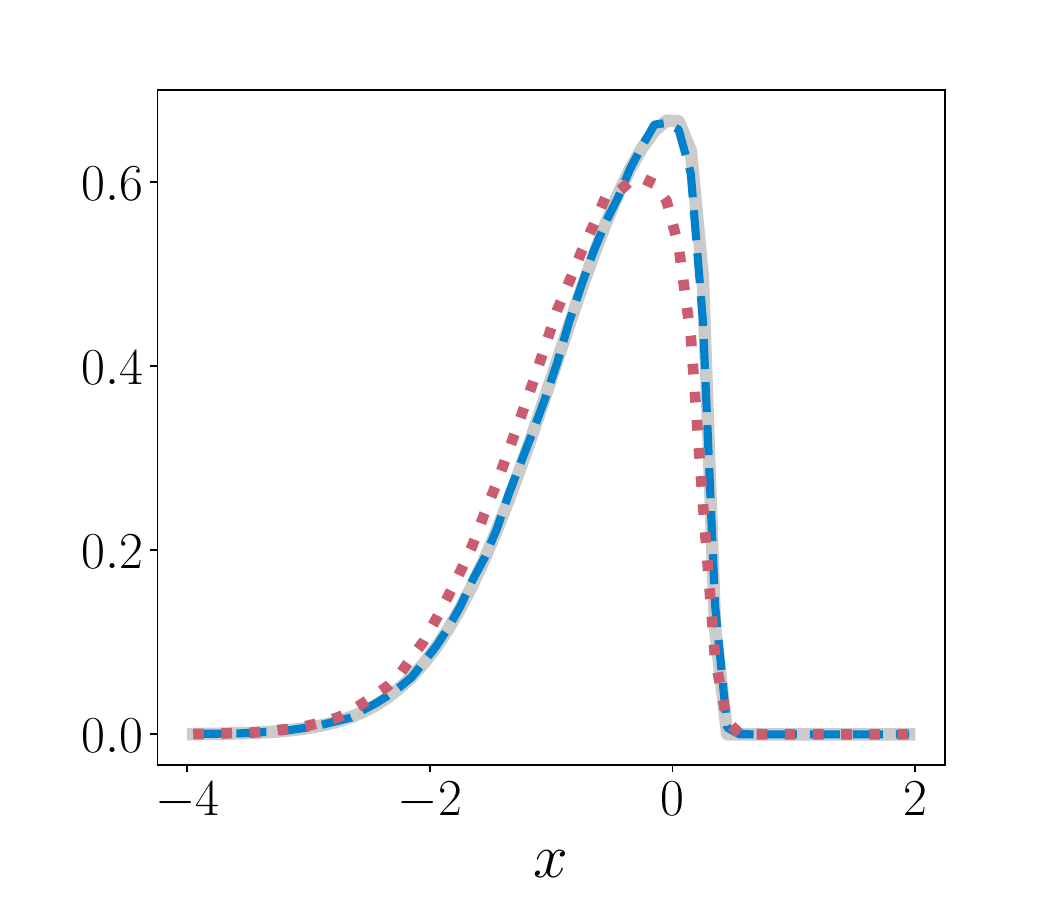}\label{fig:tr_large_h}}
\hfill
\caption{Histograms of the invariant distribution and samples from Euler-Maruyama scheme applied to the original SDE \eqref{eq:original_sde} and to the IP-transformed SDE \eqref{eq:transformed_overdamped_sde} with $n_s=10^5$ samples, using the function $g$ plotted in Figure \ref{fig:spring_right} with $M=2$, $m=0.001$, a temperature of size $\beta^{-1}=0.1$. The parameters of the potential are set to $c=0.1$, $b=0.1$, $a=10$ and $x_0=0.5$.}
\end{figure}
As shown in Figure \ref{fig:tr_small_h}, the discretized scheme \eqref{eq:em_tr_sde} allows us to sample from the invariant distribution $\rhoeq$. It not only provides samples from the desired distribution, but moreover remains accurate for larger stepsizes, as seen in Figure \ref{fig:tr_large_h}. While the processes \eqref{eq:original_sde} and \eqref{eq:transformed_overdamped_sde} are different, the values taken by $g(X_n)$ can be used as a proxy for the computational effort provided by the invariant-preserving transformed dynamics. We note that to yield the sample in Figure \ref{fig:tr_large_h}, the mean value of the monitor function $g$ over all samples and iterations is $1.513$, which suggests that simulating the SDE \eqref{eq:transformed_overdamped_sde} is computationally more efficient, i.e. a timestep of size $1.513$ would yield a similar accuracy. 

We could alternatively obtain samples from the invariant measure by using the direct time-rescaled scheme \eqref{eq:em_sde_naive} as follows:  first, for a fixed value of $T_f$ we ensure that the sum of the rescaled timesteps $g(\widetilde{X}_n)h$ taken along the trajectory reach the same final time $T_f$, rather than fixing the number of steps, where interpolation is needed to obtain the values of the sample at the precise time $T_f$. In fact, the direct time rescaled SDE \eqref{eq:naive_time_rescaled} targets 
\begin{equation} \label{eq:Gibbs_eq_dtr}
    \hat{\rho}(x) = \frac{1}{\hat{Z}} \exp \left(-\beta V(x) - \log(g(x)) \right).
\end{equation}
One can interpret this device as a form of importance sampling to approximate $\rho$ using samples \eqref{eq:naive_time_rescaled}. Here, we use the basic identity
$$
\int Q(x) \rho(x) dx = \frac{\hat{Z}}{Z} \int Q(x) g(x) \hat{\rho}(x) dx = \frac{\int Q(x) g(x) \hat{\rho}(x) dx}{\int g(x') \hat{\rho}(x') dx'},
$$
that holds for any appropriate function $Q: \Omega \rightarrow \mathbb{R}$.
We obtain an estimator in the following way
\begin{align*}
    \int Q(x) \rho(x) dx &= \lim_{T'_f \rightarrow \infty} \frac{1}{T'_f} \int_0^{T'_f} Q(\tilde{x}(t)) dt \\ &\approx  \frac{1}{T_f} \int_0^{T_f} Q(\tilde{x}(t)) dt \approx \frac{\sum_{n=0}^N Q(\tilde{X}_n)g(\tilde{X}_n)h}{\sum_{n'=0}^N g(\tilde{X}_{n'})h},
\end{align*}
where the first equality assumes an ergodic theorem for the time rescaled SDE, the  first approximation relation is due to choosing a large $T_f > 0$ instead of evaluating the limit, and the second approximation is due to the sampling error, where we define  $T_f$  such that $T_f= \sum_{n'=0}^N g(x_{n'})h$. Fixing a $T_f$ in this way corresponds to `reaching the same final time' mentioned above. Whilst this importance sampling approximation avoids computing the gradient of $g$, it is not clear in general how to choose a $g$ that both leads to an efficient sampling from $\hat{\rho}$ and an efficient importance sampling approximation (i.e., $g$ having a small $\chi^2$ divergence, see \cite{Agapiou}). We leave this problem for future investigation and concentrate in the sequel on the IP-transformed SDE \eqref{eq:transformed_overdamped_sde}.



\section{Design of monitor function}
\label{sec:monitor_func}
A smart choice of the monitor function $g$ slows down the process and increases the number of samples in regions of high solution change and accelerates it elsewhere. In this section, we give some intuition regarding the design of a ``good" monitor function, as well as criteria that  guarantee the well-posedness of the SDE.

\subsection{Stepsize bounds}
In the example with the potential \eqref{eq:pot_spring}, the function $g$ decreases the stepsize where the potential is  steep and increases it elsewhere. This allows us to gain accuracy, as it encourages trajectories to approach the equilibrium values in the numerical integration. Different designs of the monitor function $g$ are available.

The simplest choice (if not the most practical) is to base the monitor function on $G(x)=\|V'(x)\|$, where here and elsewhere in this work $\|\cdot\|$ denotes the $2$-norm. 
But simply choosing $g(x) = 1/G(x)$ is unreliable since the forces may vanish entirely or become very large, causing the stepsize either to grow precipitously or decay to zero. Instead, it is desirable to introduce a heuristic, as in \cite{HuangLeimkuhler1997}, which controls the adapted stepsize to lie within a fixed range.

Define a smooth, monotone function $\psi:\mathbb{R}_+\rightarrow [m,M]$, for $m<M$ such that:
\[
\lim_{u\rightarrow \infty} \psi(u) = m, \qquad
\lim_{u\rightarrow 0} \psi(u) = M.
\]
For this choice the composite function $g(x) = \psi(1/G(x))$ has a defined range and has the property that $g(x)$ is minimized in the limit of large $G(x)$, and is maximal when $G(x)\approx 0$. We used a similar choice to the one suggested in \cite{HuangLeimkuhler1997},
\begin{align} \label{eq:monitor_phi}
\psi(x) = \frac{\sqrt{1+m^2 r x^{2\alpha}}}{\frac{\sqrt{1+m^2 r x^{2\alpha}}}{M}+ \sqrt{r x^{2\alpha}}} 
\end{align}
where we added the dependence on the power $\alpha \in \mathbb{N}_+$ and the parameter $r \in \mathbb{R}_+$, which has more rapid decay for larger $r$, see Figure \ref{fig:heuristics}. The latter  reaches $\frac{mM}{m+M}$ at $u=0$, but if $m\ll M$, this difference is inconsequential.
\begin{figure}[htp]
\label{fig:heuristics}
\begin{center}
\includegraphics[width=5in]{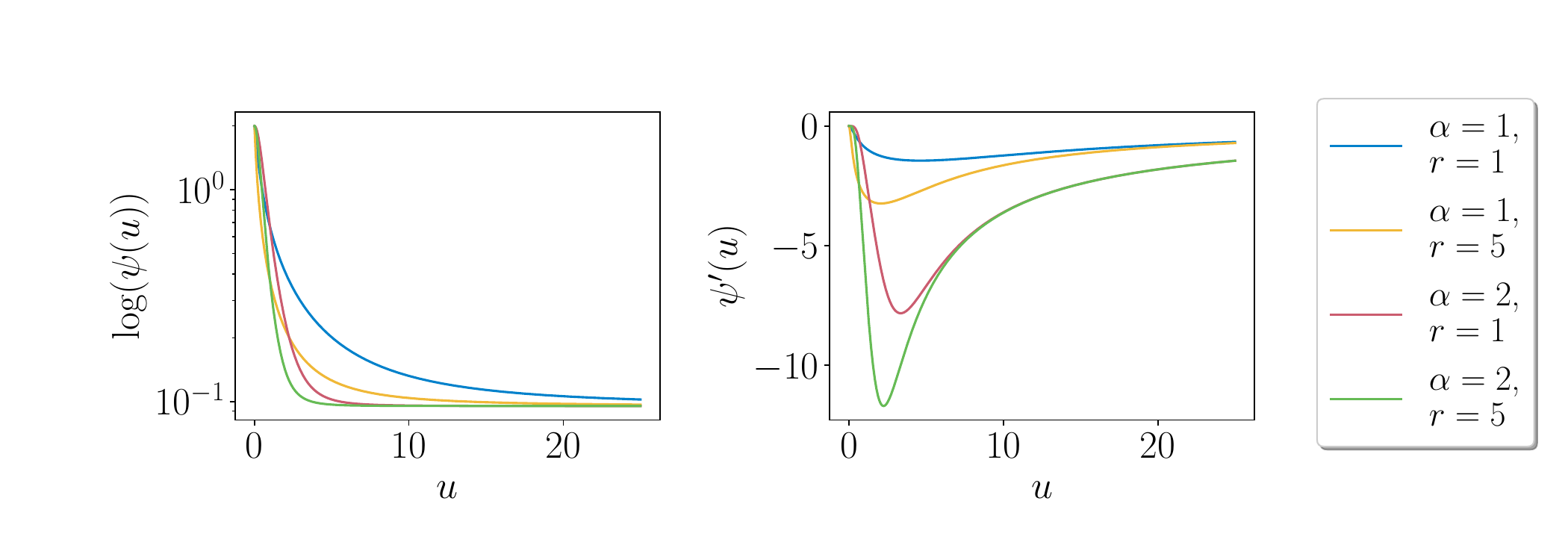}
\caption{Comparison of the monitor function for different values of the parameter $r$ and $t$.}
\end{center}    
\end{figure}
By choosing the coefficients $m,M$, we can restrict the effective stepsize $g(X_n)h$ to a range between a maximum and minimum stepsize, with the lower and upper bounds given by $m h$ and $M h$ respectively. The tuning of the parameters $\alpha$ and $r$ is problem-dependent.
Referring to the example in Section \ref{sec:running examples}, we define $g_1(x)=\psi(V'(x))$,  $g_2(x)=\psi(\omega^2(x))$ and  $g_3(x)=\psi(\omega(x))$, as shown in Figure \ref{fig:function_g}.
\begin{figure}[htp]
\hfill
\subfigure[$g_1(x)=\psi(V'(x))$.]{\includegraphics[width=4cm]{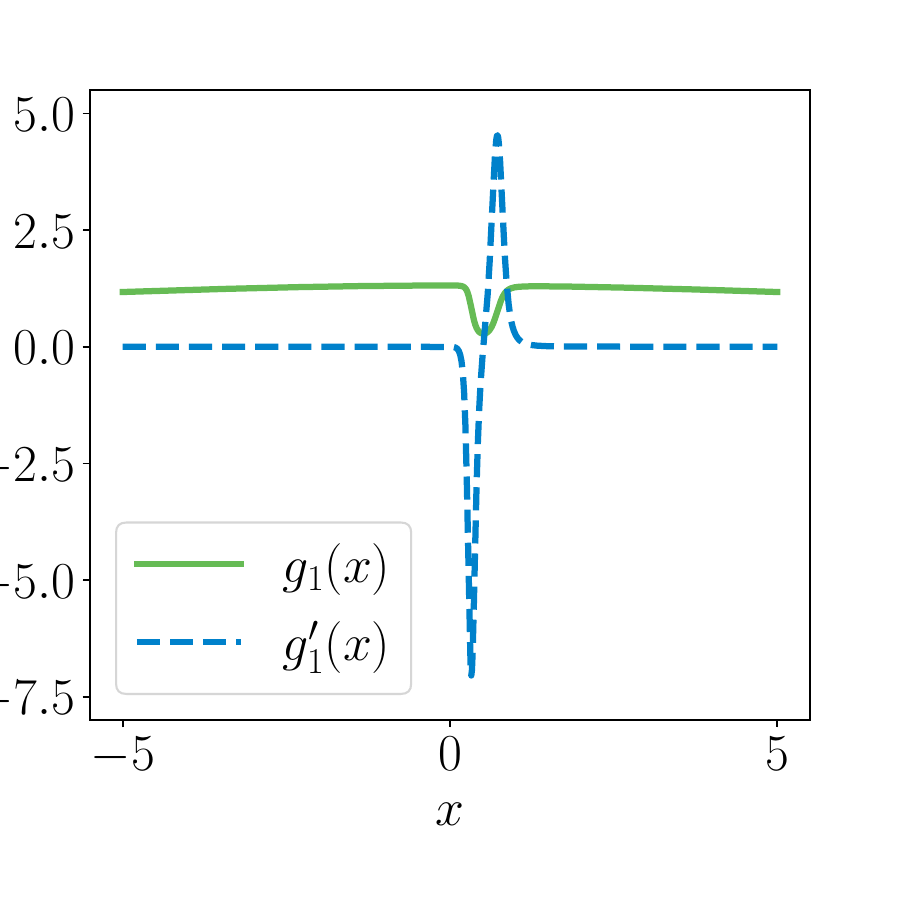}}
\hfill
\subfigure[$g_2(x)=\psi(\omega^2(x))$ .]{\includegraphics[width=4cm]{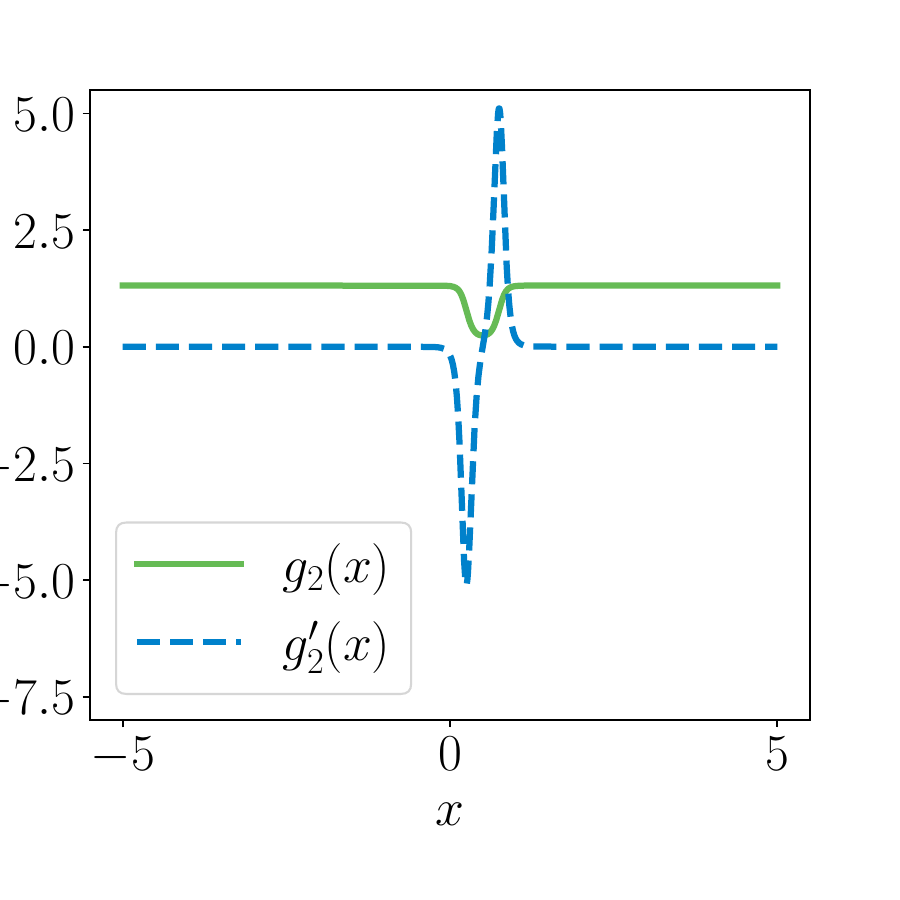}}
\hfill
\subfigure[$g_3(x)=\psi(\omega(x))$.]{\includegraphics[width=4cm]{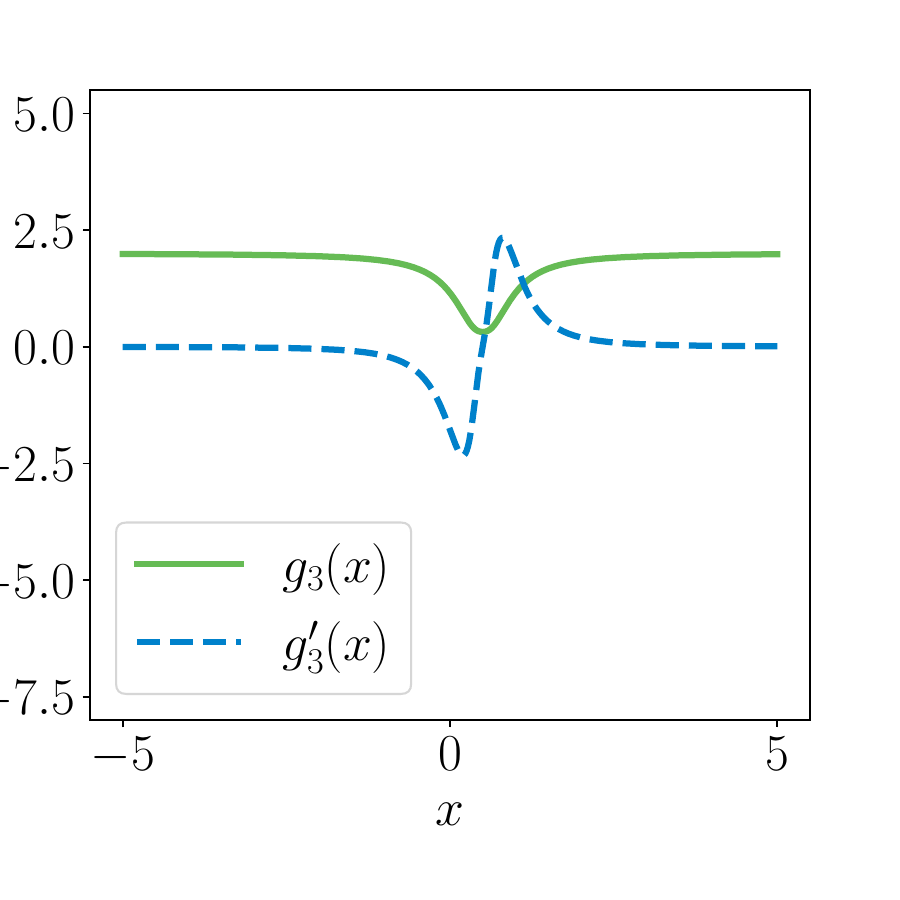}}
\hfill
    \caption{Different designs of the function $g$, and the corresponding function $g'$.}
    \label{fig:function_g}
\end{figure}
In Figure \ref{fig:function_g}, we plot the functions $g_1$, $g_2$ and $g_3$, respectively from left to right. We observed that the numerical integration remained stable for larger stepsize with the function $g_3$. This is due to the smaller values taken by $g_3'$, allowing us to obtain comparable results with a large range of stepsizes and providing 
informative accuracy curves (presented in Section \ref{sec:numerical_experiments}). 
More rapid changes of the function $g$ will induce higher values of $\nabla g$, which means that smaller steps will be required for the numerical integration. There is a balance to strike; a highly adaptive monitor function $g$ can provide better efficiency but will induce a large correction term that might impair the stability of the SDE. Intuitively, 
it may be argued that we are attempting to ease the integration of the SDE with drift terms having a large Lipschitz constant by introducing a taming multiplicative factor in the form of a monitor function. 

With a good design of monitor function, the numerical integrators reach their asymptotic state with lower computational effort, which is equivalent to a larger average value taken by the monitor function, or, mean stepsize. With the original systems \eqref{eq:full_langevin_1d} and \eqref{eq:original_sde}, larger stepsizes do not allow the numerical integration to weakly approach the target distribution.
In the general non-scalar case, without detailed knowledge of the problem, a simple heuristic monitor function could be taken to be $\psi(\|\nabla V(x)\|)$, but there are often better choices.  For example, if the computation of the potential is very expensive, one may use a simplified, more computable measure of its magnitude. An example arises in Bayesian parameter inference when the prior distribution is steeply confined but the model likelihood is smooth, typically more expensive to compute, see Section \ref{sec:bayesian}. An other example arises in $N$-body simulation where interactions with distant bodies contribute little to the overall potential or gradient norm; these long-range terms can be neglected in calculating the monitor function, saving considerable computation. When the problem is too complex, or the dynamics is not accessible, the monitor function can be estimated on the fly (using history) or one may be able to use the Hessian as in Newton's method \cite{NocedalJorge2006}. 

\subsection{Well-posedness of the IP-transformed SDE}
\label{sec:wellposedness}
From \cite[Theorem~5.2.1]{Oksendal2003} an SDE with drift term $b:[0,\infty[ \times \mathbb{R}^d \to \mathbb{R}^d$ and diffusion $b:[0,\infty[ \times \mathbb{R}^d \to \mathbb{R}^d$ has a unique solution $x$ that is continuous with respect to time under global Lipschitz conditions
\begin{subequations}\label{eq:cdts_th_oksen_2}
    \begin{align}\label{eq:cdts_th_oksen_2_1}
    \|b(t,x)-b(t,y)\|&\leq  C_1 \|x-y\|;\ x,y \in \mathbb{R}^d \\
    \label{eq:cdts_th_oksen_2_2}
    \|\sigma(t,x)-\sigma(t,y)\|&\leq C_2 \|x-y\|; \ x,y \in \mathbb{R}^d.
\end{align}
\end{subequations}
for some constant $C_1,C_2$. We now derive criteria that may be imposed on the monitor function to ensure that these conditions are satisfied for the IP-transformed SDE. 

We make the following assumptions on the potential.
\begin{assumption} \label{ass:pot_1}
The potential $V$ is smooth and confining, hence the Gibbs distribution \eqref{eq:Gibbs_eq} is the unique invariant distribution of the SDE \eqref{eq:original_sde}.
\end{assumption}
\begin{assumption} \label{ass:pot_2}
The potential is a measurable function satisfying
   \begin{align}
\|\nabla V(y) - \nabla V (x)\|\leq D \|x-y\|; \ x,y \in \mathbb{R}^d
\end{align}
for some constant D, which implies
\begin{align} \label{eq:ass_pot_2_lineargrowth}
    \|\nabla V(x)\|\leq C(1 + \|x\|);\ x \in \mathbb{R}^d
\end{align}
for some constant $C$.
\end{assumption} 


The IP-transformed SDE \eqref{eq:transformed_overdamped_sde} has a drift term $\tilde{b}(x)=-\nabla V(x) g(x) + \beta^{-1} \nabla g(x)$ and a diffusion term $\tilde{\sigma}(x) = \sqrt{2 \beta^{-1} g(x)}$. 
The monitor function $g$ is assumed to be differentiable, and we enforce a global Lipschitz condition on this function and its gradient. 
\begin{criterion} \label{gLipschitz}
   There exists a constant $C_1>0$ such that
    \begin{equation}
        |g(x)-g(y)| \leq C_1 \|x-y\|; \ x,\ y \in \mathbb{R}^d.
    \end{equation}
\end{criterion}
\begin{criterion} \label{nablagLipschitz}
The function $\nabla g$ exists and there exists a constant $C_2>0$ such that
    \begin{equation}
        \|\nabla g(x)- \nabla g(y)\|\leq C_2 \|x-y\|; \ x,y \in \mathbb{R}^d.
    \end{equation}
\end{criterion}
In practice, the function $g$ is built to re-scale the value of the stepsize and, thus, is bounded. 
\begin{criterion} \label{gbounded}
The function $g:\mathbb{R}^d \to \mathbb{R}$ is such that there exist two strictly positive constants $M_1,\ m_1$, with $m_1<M_1$, such that
\begin{equation}
    M_1>g(x)>m_1 ; \ x \in \mathbb{R}^d.
\end{equation}
\end{criterion}
We note that Criterion \ref{gLipschitz} implies that the drift term $\sqrt{2 \beta^{-1} g }$ is Lipschitz as $\sqrt{x}$ is Lipschitz on $(m_1, \infty)$ and Criterion \ref{gbounded} ensures that $g(x) >m_1;\ x \in \mathbb{R}^d$. Thus, the diffusion term $\tilde{\sigma}(x) = \sqrt{2 \beta^{-1} g(x)}$ satisfies \eqref{eq:cdts_th_oksen_2_2}. 
Finally, we require that the product $\nabla V g $ is Lipschitz. The function $g$ is smooth and bounded, and the potential is also smooth and confining. The function $g$ may be defined to be constant outside a 
suitable ball, so our final criterion holds by design.
\begin{criterion} 
\label{VgLipschitz}
There exists a constant $C_3>0$ such that
    \begin{equation}
        \| \nabla V(x) g(x)- \nabla V(y) g(y)\|\leq C_3 \|x-y\|; \ x,y \in \mathbb{R}^d.
    \end{equation}
\end{criterion}
Then, by the triangle inequality, the Criteria \ref{gLipschitz} and \ref{VgLipschitz} ensure that the drift term $\tilde{b}(x)=-\nabla V(x) g(x) + \beta^{-1} \nabla g(x)$ respects \eqref{eq:cdts_th_oksen_2_1}. The desired proposition follows.
\begin{proposition} \label{prop:uniquesol}
If $b(x) \equiv - \nabla V(x)$ in the SDE \eqref{eq:original_sde} 
satisfies  condition \eqref{eq:cdts_th_oksen_2_1}  
then, under  Assumptions \ref{ass:pot_1} and \ref{ass:pot_2} and Criteria \ref{gLipschitz} - \ref{VgLipschitz}, the IP-transformed SDE \eqref{eq:transformed_overdamped_sde} 
satisfies 
 condition \eqref{eq:cdts_th_oksen_2_1} and  \eqref{eq:cdts_th_oksen_2_2}
and hence has a unique solution that is continuous with respect to time.
\end{proposition}
\section{Invariant distribution and geometric ergodicity}
\label{sec:correct_timestepping} 
In Section \ref{sec:running examples}, our proposed correction in \eqref{eq:transformed_overdamped_sde} is necessary to ensure stationarity with respect to the correct distribution and in Section \ref{sec:wellposedness}, we showed there exist a unique time continuous process for equation \eqref{eq:transformed_overdamped_sde}. In the following section, using uniform ellipticity of the Fokker-Planck equation, we show that the Gibbs distribution is the unique invariant distribution to equation \eqref{eq:transformed_overdamped_sde}. Additionally, we introduce the transformed version of the underdamped Langevin equations \eqref{eq:full_langevin_1d} and provide results in this case which correspond to those obtained for overdamped systems. The underdamped system is, in principle, more challenging but we outline a simpler approach here that avoids invoking a complicated analysis.

\subsection{The overdamped Langevin dynamics}
We follow Theorem 4.1 \cite[p~89]{Pavliotis2014}, which applies to the time homogeneous process $x$ in equation \eqref{eq:original_sde}
where the diffusion matrix is the product of the diffusion term $\sigma(x)=\sqrt{2 \tau}$, i.e $\Sigma(x)=\sigma(x) \sigma(x)^T: \mathbb{R}^{d\times d} \to \mathbb{R}^{d\times d}$, with an initial condition $X_0$ with probability density function $\rho_0(x)$. The solution $u(x,t) \in C^{2,1}(\mathbb{R}^d \times \mathbb{R}^+)$ is the solution to an initial value problem with a uniform elliptic backward Kolmogorov equation of the form
\begin{equation*}
    \partial_t u =  \tfrac{1}{2} \sum_{i,j=1}^{d}  \Sigma_{ij}(x) \partial^2_{x_i,x_j} u + \tilde{b}(x) \cdot \nabla u + c(x) u.
\end{equation*} If we have uniform ellipticity on the diffusion, there exists a constant $\alpha >0$  such that 
\begin{equation} \label{eq:the4elli}
    \langle \xi, \Sigma(x) \xi \rangle \geq \alpha \| \xi \|^2,\ \forall \xi \in \mathbb{R}^d,
\end{equation} uniformly in $x \in \mathbb{R}^d$ which is trivial for \eqref{eq:original_sde}. If uniform ellipticity holds and if the coefficients $\tilde{b}(x)=\nabla V(x)+ \nabla \cdot \Sigma(x)$ and $c(x)= \nabla \cdot \nabla V(x) + \frac{1}{2} \nabla \cdot (\nabla \cdot \Sigma)(x)$ are smooth and satisfy the growth conditions
\begin{subequations}\label{eq:the4grwthcdts}
\begin{align}
    &  \label{eq:the41_a} \| \Sigma(x) \|\leq B_1,\\
   & \label{eq:the41_b} \| \tilde{b}(x) \|\leq B_2 (1+\|x\|),\\
    & \label{eq:the41_c} \| c(x) \|\leq B_3 (1+\|x\|^2),
    \end{align}
\end{subequations}
then there exists a unique solution to the Cauchy problem for the Fokker Planck equation. For equation \eqref{eq:original_sde}, conditions \eqref{eq:the41_a} follows from the constant diffusion term and \eqref{eq:the41_b} is equivalent to the linear growth assumption in equation \eqref{eq:ass_pot_2_lineargrowth}. The following assumption is made on the potential to ensure \eqref{eq:the41_c}.
\begin{assumption}\label{ass:pot_3}
    There exist a constant $C_4>0$ such that
    \begin{equation}
        \| \sum_{i=1}^d \partial_{x_i}^2 V(x) \|\leq C_4(1+\|x\|^2); x \in \mathbb{R}^d.
    \end{equation}
\end{assumption}
We show that Theorem~4.1 holds for the transformed equation \eqref{eq:transformed_overdamped_sde}, given two additional criterions on the monitor function.
\begin{criterion} \label{nablagbounded}
The function $\nabla g:\mathbb{R}^d \to \mathbb{R}^d$ is such that there exist two strictly positive constants $M_2, M_3$ such that for all $x$
\begin{equation}
    M_2>\|\nabla g(x)\|; \ x \in \mathbb{R}^d,
\end{equation}
and for $j=1, \dots, d$, we have
\begin{equation}
    M_3>|\partial_j g(x)|; \ x \in \mathbb{R}^d.
\end{equation}
\end{criterion}
and
\begin{criterion} \label{ddgbounded}
    The function $ g:\mathbb{R}^d \to \mathbb{R}^d$ is twice differentiable and has bounded second derivative. For $i,j=1, \dots, d$, we have
    \begin{equation}
    M_4>|\partial^2_{x_i} g(x)|; \ x \in \mathbb{R}^d.
\end{equation}
\end{criterion}
we can 
\begin{theorem}
If the SDE \eqref{eq:original_sde} satisfies the Assumptions \ref{ass:pot_1} to \ref{ass:pot_3} and provided that the monitor function follows Criterions \ref{gLipschitz} to \ref{ddgbounded}, the IP-transformed SDE \eqref{eq:transformed_overdamped_sde} 
\begin{align*}
d x & = -g(x)\nabla V(x) d t  + \beta^{-1}  \nabla g(x) d t + \sqrt{2\beta^{-1} g(x)} d W 
\end{align*}
has the unique invariant distribution $\rhoeq$, as given in \eqref{eq:Gibbs_eq}.
\end{theorem}
\begin{proof}
First, we demonstrate that the distribution \eqref{eq:Gibbs_eq} is invariant with respect to the SDE \eqref{eq:transformed_overdamped_sde}. It can be see by examining the adjoint operator of the infinitesimal generator associated with equation \eqref{eq:transformed_overdamped_sde}
\begin{equation*} 
        \mathcal{L}^* \rho(t,x) =- \nabla \cdot \left( -\nabla V(x)g(x) + \beta^{-1} \nabla g(x) \right)\rho(t,x) + \beta^{-1} \Delta \left(g(x)  \rho(t,x)\right) .
\end{equation*}
The invariant distribution \eqref{eq:Gibbs_eq} is a stationary density if it is in the kernel of the adjoint operator $\mathcal{\tilde{L}}^*$. We have
\begin{equation*}
\mathcal{L}^* \rhoeq =- \nabla \cdot \left( -\nabla V(x)g(x) + \beta^{-1} \nabla g(x) \right) \rhoeq + \beta^{-1} \nabla \cdot \left( \nabla g(x)  \rhoeq + g(x) \nabla \rhoeq \right),
\end{equation*}
and we use consecutively the chain rule and distribute the gradient to expand the last term, and since $g$ is scalar valued, we have
\begin{equation*}
\begin{split}
      \mathcal{L}^* \rho(t,x) &= \nabla \cdot g(x) \nabla V(x) \rho(t,x) - \nabla \cdot g(x) \nabla V(x) \rho(t,x)   \\&\quad - \beta^{-1} \nabla \cdot \nabla g(x) \rho(t,x) + \beta^{-1}  \nabla \cdot \nabla g(x)  \rho(t,x)\\
  &= 0,
  \end{split}
\end{equation*}
and the probability distribution \eqref{eq:Gibbs_eq} is in the kernel of the adjoint operator $\mathcal{\tilde{L}}^*$ and is an invariant distribution of the SDE \eqref{eq:transformed_overdamped_sde}. Second, we prove uniqueness of the solution by showing that conditions \eqref{eq:the4elli} and \eqref{eq:the4grwthcdts} hold for the transformed dynamics\eqref{eq:transformed_overdamped_sde}. The diffusion matrix $\hat{\Sigma}(x)= 2 \beta^{-1} g^2(x) \mathrm{I}$ of the SDE \eqref{eq:transformed_overdamped_sde} respects \eqref{eq:the4elli} and \eqref{eq:the41_a} by Criterion \ref{gbounded}. We note that for the IP-transformed SDE \eqref{eq:transformed_overdamped_sde}, the condition \eqref{eq:the41_b} is applied on $\| \hat{b}(x) \| = \| \nabla V(x) g(x) +\beta^{-1} \nabla g(x) + \nabla \cdot \sqrt{2 \beta^{-1} g(x)} \mathrm{I} \|$. Using the triangle inequality and computing the divergence of the last term
\begin{equation*}
    \| \hat{b}(x) \|\leq \|\nabla V(x) g(x)\| + \| \beta^{-1} \nabla g(x) \| + \frac{\sqrt{2 \beta^{-1}}}{2} \bigg \|\frac{\nabla g(x)}{\sqrt{g(x)}}\bigg \|,
\end{equation*}
noting that $\nabla \cdot g(x) \mathrm{I} = \nabla g(x)$ and $\nabla \cdot (\sqrt{\beta^{-1} 2 g(x)} \mathrm{I} )= \frac{\sqrt{2 \beta^{-1}}}{2 \sqrt{g(x)}} \nabla g(x)$.
We can now use the Cauchy–Schwarz inequality on the first and last term, while bounding the second term using Criterion \ref{nablagbounded}. 
\begin{equation*}
\| \hat{b}(x) \|\leq \|\nabla V(x)\| |g(x)| + \beta^{-1} M_{2}+ \frac{\sqrt{2 \beta^{-1}}}{2} \bigg \|\frac{1}{\sqrt{g(x)}}\bigg \|  M_{2}.
\end{equation*}
The first term is bounded by linear growth through Assumption \ref{ass:pot_2}, and the monitor function is bounded by Criterion \ref{gbounded}. The last term is bounded by a constant through Criteria \ref{gbounded} and \ref{nablagbounded}, which implies that \eqref{eq:the41_b} holds.  
Finally, the last condition \eqref{eq:the41_c} for $\| c(x) \|\leq M(1+\|x\|^2)$ is applied to $ \| c(x) \|= \| - \nabla \cdot \left( b(x) g(x) + \beta^{-1} \nabla g(x) \right) + \frac{1}{2} \nabla \cdot \nabla \cdot \sqrt{2 \beta^{-1} g(x)} \mathrm{I} \|$. Using the triangle inequality, each term can be bounded separately, and we have
\begin{align*}
    \| c(x) \|\leq & \bigg |\sum_{i=1}^d \left(-\partial_{x_i}^2 V(x) \right)g(x) + b(x) \partial_{x_i}g(x) \bigg | + \beta^{-1} \bigg | \sum_{i=1}^d \partial^2_{x_i} g(x) \bigg | \\ &+ \frac{\sqrt{2 \beta^{-1}}}{2} \bigg | \sum_{i=1}^d \left( \frac{\partial_{x_i}^2 g(x)}{\sqrt{g(x)}} - \frac{1}{2} \frac{(\partial_{x_i} g(x))^2 }{\sqrt{g(x)}^3}\right) \bigg |.
\end{align*}
The first term is bounded by respectively the linear and quadratic growth assumptions on the potential and its divergence (i.e Assumption \ref{ass:pot_2} and \ref{ass:pot_3}) and  Criteria \ref{gbounded} and \ref{nablagbounded} on the monitor function. The second term is bounded by a constant through Criterion \ref{nablagbounded}. The third term is bounded by a constant through the Criteria \ref{gbounded}, \ref{nablagbounded} and \ref{ddgbounded}. Thus, we have that condition \eqref{eq:the41_c} holds as 
\begin{equation*}
     \| c(x) \|\leq M (1+\|x\|^2)
\end{equation*}
in the transformed case. Thus all conditions are respected and the invariant distribution is unique by Theorem~4.1.
\end{proof}

\subsection{Underdamped Langevin dynamic}
The underdamped Langevin dynamics \eqref{eq:full_langevin_1d}, under Assumption \ref{ass:pot_1}, has the unique invariant distribution
\begin{align} \label{eq:invariant_under}
   \rhoequn(x,p) = \frac{1}{Z}\exp(-\beta H(x,p)), 
\end{align}
where $H$ is the Hamiltonian $H(x,p)  = |p|^2/2 + V(x)$ and the normalizing factor $Z$ is
\begin{align} \label{eq:invariant_under_z}
    Z= \int_{\mathbb{R}^{2d}} \exp(- \beta H(x,p)) d x d p,
\end{align}
 by \cite[Prop.~6.1]{Pavliotis2014}. We have the following relationships for the invariant distribution
\begin{subequations}
\begin{equation} \label{eq:nablaqrho}
\nabla_{x} \rhoequn =- \beta \nabla_{{x}} V({{x}}) \rhoequn, 
\end{equation}
\begin{equation}\label{eq:nablaprho}
\nabla_p \rhoequn =- \beta p \rhoequn.
\end{equation}
\end{subequations}

We introduce a time-rescaling together with a modification of the force which preserves the invariant measure to yield the IP-transformed underdamped dynamics. 
\begin{theorem}
When Assumptions \ref{ass:pot_1} and \ref{ass:pot_2} and Criteria \ref{gLipschitz} to \ref{VgLipschitz} hold, the system of SDEs
\begin{align} \label{eq:th_transformed_underdamped}
\begin{cases}
d x & = g(x) p d t,\\
d p & = -g(x)\nabla_x V(x) d t  + \beta^{-1} \nabla_x g(x) d t - \gamma g(x)pd t + \sqrt{2\gamma \beta^{-1} g(x)} d W(t)
\end{cases}
\end{align}
has the invariant distribution \eqref{eq:invariant_under} 
\begin{align}
   \rhoequn(x,p) = \frac{1}{Z}\exp(-\beta H(x,p)), 
\end{align}
with normalizing constant \eqref{eq:invariant_under_z} \begin{align} 
    Z= \int_{\mathbb{R}^{2d}} \exp(- \beta H(x,p)) d x d p.
\end{align}
\end{theorem}

\begin{proof}
To show that the equation \eqref{eq:invariant_under} is a solution to the system \eqref{eq:th_transformed_underdamped}, we introduce the initial value problem for the forward Kolmogorov equation associated to the system \eqref{eq:th_transformed_underdamped} with transition density $u(t,x,p) \in C^{2,1}(\mathbb{R}^+ \times \mathbb{R}^d \times \mathbb{R}^d)$. This is given by
\begin{align}\label{eq:operator_FP_inva}
\begin{split}
\mathcal{L^*} u =& - \begin{pmatrix} \nabla_x\\ \nabla_p\end{pmatrix} \cdot \begin{pmatrix} p g(x) \\ -(\nabla_x V(x))g(x)+\beta^{-1} \nabla_x g(x)  - \gamma g(x) p \end{pmatrix} u \\&+ \frac{1}{2}  \begin{pmatrix} \nabla_x\\ \nabla_p\end{pmatrix} \cdot \begin{pmatrix} \nabla_x\\ \nabla_p\end{pmatrix} \cdot \left( \begin{pmatrix} 0&0\\0& \sqrt{2 g(x) \frac{\gamma}{\beta}} \end{pmatrix}\begin{pmatrix} 0&0\\0& \sqrt{2 g(x) \frac{\gamma}{\beta}} \end{pmatrix}^T u \right)
\end{split}
\end{align}
which yields
\begin{align}
\begin{split}
 \mathcal{L^*} u =&- \nabla_x (g(x)p) u - \nabla_p \left[ -g(x) (\nabla_x V(x))+\beta^{-1} \nabla_x g(x) -\gamma g(x) p\right] u \\&+ \gamma \beta^{-1} g(x) \Delta_p  u.
\end{split}
\end{align}
Replacing $u$ by $\rhoequn$ and using \eqref{eq:nablaprho}, the last two terms cancel as the configuration dependence of $g(x)$ plays no role, i.e., it corresponds to an Ornstein-Uhlenbeck SDE which automatically preserves $\rhoequn$ $-\Delta_p (\gamma g(x) p) \rhoequn = \gamma \beta^{-1} g(x) \Delta_p  \rhoequn$. Using the product rule in the first term and \eqref{eq:nablaqrho}, as well as  \eqref{eq:nablaprho} on the second and third term, we are left with
\begin{align*}
\begin{split}
{\cal L}^* \rhoeq =& -p \left(\nabla_{x} g(x)\right) \rhoeq + \beta p g(x) \nabla_x V(x) \rhoeq + p \left(\nabla_x g(x) \right) \rhoeq - \beta p g(x) \nabla_x V(x) \rhoeq 
\end{split}
\end{align*}
implying that ${\cal L}^* \rhoeq = 0$. 
\end{proof}

Since the noise only appears in the momentum equation, the second order term depends only on the momentum $p$, and we cannot use a similar reasoning for the system \eqref{eq:th_transformed_underdamped}, but hypoellipticity of the operator can be shown using the H\"ormander condition to establish uniqueness of the solution \cite{HormanderLars1967Hsod}. A discussion can be found in \cite[Section~6.1]{Pavliotis2014}. 

An alternative intuitive reasoning arises by taking the first moments with respect to $\rhoequn$ of the process $X_t$ associated to the system \eqref{eq:th_transformed_underdamped}. The expression $\lim_{t^* \to \infty}\frac{1}{t^*} \int_0^{t^*} f(X_t) d t$ under the change of variable $d t = g^{-1}(X_{\tau})  d \tau$ is 
\begin{align}\label{eq:frac_firstmom}
\lim_{\tau(t^*) \to \infty}\frac{ \frac{1}{\tau(t^*)}\int_0^{{\tau(t^*)}} f(X_\tau) g^{-1}(X_\tau) d\rhoequn (\tau)}{ \frac{1}{\tau(t^*)} \int_0^{\tau(t^*)} g^{-1}(X_\tau) d\rhoequn (\tau)}.
\end{align}
Under a similar change of variable $d \tau = g(X_{\tau})  dt$, it is easy to show that the system \eqref{eq:th_transformed_underdamped} becomes
\begin{align} \label{eq:tr_under_timechanged}
\begin{cases}
d x & = p d \tau,\\
d p & = \left(-\nabla_x V(x)+ \beta^{-1} \nabla_x \log (g(x)) \right) d \tau - \gamma p d \tau + \sqrt{2\gamma \beta^{-1}} d W(\tau), 
\end{cases}
\end{align}
with invariant distribution
\begin{equation*}
     \hatrhoequn= \hat{Z} \exp\left(-\beta V(x) + \log g(x) + \frac{\|p\|^2}{2} \right).
\end{equation*}
We now show uniqueness of this invariant distribution and geometric ergodicity of \eqref{eq:tr_under_timechanged}.
Indeed, these follow from Theorem 3.2 under Condition 3.1 from \cite{MattinglyStuartHigham2002}. This condition requires that for such an underdamped system \eqref{eq:tr_under_timechanged}, we have that $F(x)=\left( V(x)- \beta^{-1} \log (g(x)) \right)$ is such that $F(x) \geq 0,$ for all $x \in \mathbb{R}^d$, and that there exists an $\alpha > 0 $ and a $\beta \in (0,1)$ such that 
\begin{equation}
    \frac{1}{2} \langle  F(x),x \rangle \geq \beta F(x) + \gamma^2 \frac{\beta (2-\beta)}{8(1-\beta)} \|x\|^2 - \alpha 
\end{equation}
which applies if $F$ grows at infinity as $\|x\|^{2l}$, where $l$ is some positive integer. This condition is true by Assumption \ref{ass:pot_1} and by $\log(g(x))$ being bounded by Criterion \ref{gbounded}. Thus, the dynamical system \eqref{eq:tr_under_timechanged} is geometrically ergodic, the invariant distribution is unique and the expression \eqref{eq:frac_firstmom} can be written as 
\begin{align*}
 \frac{\int_{\Omega}  \exp(-\beta V(x) + \log g(x)) f(x) g^{-1}(x) d \rhoequn(x)}{\int_{\Omega}  \exp(-\beta V(x) + \log g(x)) g^{-1}(x) d \rhoequn(x)}.
\end{align*}
Under the usual change of variable, we recover the space average of the system \eqref{eq:th_transformed_underdamped}
\begin{align*}
\frac{\int_\Omega  \exp(-\beta V(x)) f(x) d \rhoequn(x)} {\int_\Omega  \exp(-\beta V(x)) d \rhoequn(x)}
\end{align*}
using $\exp\left(\log g(x)\right)g^{-1}(x)=1$. This implies ergodicity of the continuous process \eqref{eq:th_transformed_underdamped}:
\begin{align*}
 \lim_{t^* \to \infty}\frac{1}{t} \int_0^{t^*} f(X_t) d \rhoequn(t) = \int_\Omega  \exp(-\beta V(x)) f(x) d \rhoequn(x).
\end{align*}



\section{Numerical integrators}
\label{sec:numerical_cons}
Having shown that the continuous processes have the unique Gibbs distribution as invariant distribution, we turn our focus to the design of numerical methods for \eqref{eq:transformed_overdamped_sde} and \eqref{eq:th_transformed_underdamped}. Regarding the overdamped system, a sufficient condition to guarantee that the numerical approximation of the long term behaviour of the stochastic differential equation converges to the invariant distribution is provided by \cite[Theorem~5.1, 5.2, 5.3]{MattinglyStuartTretyakov2010}. For the underdamped system \eqref{eq:full_langevin_1d}, a simple and popular approach to numerical timestepping integration is based on splitting schemes \cite{BOU-RABEENAWAF2010LAOV,MelchionnaSimone2007Doqp,SKEELROBERTD.2002Aiif}, which break the equations into separate parts to be solved independently. Following the approach of \cite{LeimkuhlerMatthews2013,LeMaSt2015}, we use splitting schemes which are built from the building blocks A, B and O as illustrated below:
\begin{align} \label{eq:full_langevin_1d_split}
    d \begin{pmatrix}
        x \\ p
    \end{pmatrix} =\underbrace{\begin{pmatrix}
        p \\ 0
    \end{pmatrix} dt}_\text{A}+ \underbrace{\begin{pmatrix}
        0 \\ -\nabla V(x)
    \end{pmatrix} dt}_\text{B} +\underbrace{\begin{pmatrix}
        0\\ -\gamma p dt + \sqrt{2 \beta^{-1}} d W(t)
    \end{pmatrix}}_\text{O}.
\end{align}
Letter sequences denote numerical methods so for example the operator associated to the method ABO is
\begin{align}
    \Psi_{\textrm{ABO}}^{h}=\Psi_{\textrm{O}}^{h} \circ \Psi_{\textrm{B}}^{h} \circ
    \Psi_{\textrm{A}}^{h},
\end{align}
where  $\Psi_{\textrm{A}}^{h} =\left ( x +h p,p \right )$,  $\Psi_{\textrm{B}}^{h} =\left ( x, p-h \nabla V(x) \right )$ and for the O step we use the map
\[
\Psi_{\textrm{O}}^{h}(x,p) = \left ( x, \exp(-h\gamma)p + \sqrt{\beta^{-1} (1-\exp(-2\gamma h))} Z \right )
\]
with $R \sim \mathcal{N}(0,1)$. 
In the case of symmetric composition methods such as ABOBA, OBABO, etc., we view each of the  A and B operations as being performed with a half time step.  Each component can be identified with an associated generator $\mathcal{L}_{\textrm{A}}$, $\mathcal{L}_{\textrm{B}}$,  and $\mathcal{L}_{\textrm{O}}$, where, for example,  
\[
\Psi_{\textrm{A}}^{h} \equiv \exp\left (h {\mathcal{L}}_{\textrm{A}} \right).
\]

As in \cite{LeimkuhlerMatthews2013,LeMaSt2015} we write the discrete propagator in the form $\exp(t {\mathcal{L}})$ using a perturbation series
\begin{align*}
    \hat{\mathcal{L}} = \mathcal{L}_{\textrm{LD}} + h \hat{\mathcal{L}}_1 + h^2 \hat{\mathcal{L}}_2 + O(h^4),
\end{align*}
and employ the Baker-Campbell-Hausdorff (BCH) formula to work out the terms of the expansion \cite{HairerErnst2006GNIS}.  The probability distribution associated to the discretization scheme can be assumed to have a density which evolves from timestep $n$ to timestep $n+1$ by
\[
\rho_{n+1} = \exp (h {\hat{\mathcal L}^*}) \rho_n,
\]
where $\hat{\mathcal{L}}^*$ represents the $L_2$-adjoint of $\hat{\mathcal{L}}$.  In a similar way, BCH can be used to work out the terms of the adjoint of the generator. One useful property is that symmetric composition methods constructed in this way have even weak order, due to cancellation properties of the BCH expansion.  In a typical case such symmetric schemes are found to be more efficient (greater accuracy per unit computational work) than their asymmetric counterparts, since we can often reuse a force evaluation performed at the end of one step at the start of the next.    More details on the design and analysis of such splitting schemes may be found in \cite{LeMaSt2015}. 

Inspired by such splittings, we seek similar methods for the transformed system of SDEs \eqref{eq:th_transformed_underdamped} which requires the computation of extra terms such as $\nabla g(x) \beta^{-1}$, as well as rendering the step A implicit. If the correction term is included in Step B, we have 
\begin{align}
\hat{\rm{B}}&:=
\begin{cases}
dx&=0\\
dp &= -\nabla V(x(t)) g(x(t)) dt + \nabla g(x) \beta^{-1} dt
\end{cases}\\
\hat{\rm{O}}&:=
\begin{cases}
dx&=0\\
 d p &=-\gamma p g(x)dt + \sqrt{2\beta^{-1}  g(x)}dW(t).
 \end{cases}
\end{align}
Both of these vector fields can be evolved exactly, in the weak sense, with their solutions summariced in Algorithms
\ref{alg:stepB_hat} and \ref{alg:stepO}.
\begin{algorithm}[H]
\caption{A step of size $h$ for $\hat{\rm{B}}$.  Inputs: $X,\ P,\ h,\ F,\ G,\ G_p$.} \label{alg:stepB_hat}
\begin{algorithmic}
\STATE{$G \leftarrow g(X)$}
\STATE{$F \leftarrow -\nabla V(X)$}
\STATE{$G_p \leftarrow \nabla g(X)$}
\STATE{$P \leftarrow P+h(G F+\beta^{-1} G_p)$} 
\RETURN $P,\ G,\ F,\ G_p$
\end{algorithmic}
\end{algorithm}
\begin{algorithm}[H]
\caption{A step of size $h$ for $\hat{\rm{O}}$.  Inputs: $X,\ P,\ h,\ F,\ G$.} \label{alg:stepO}
\begin{algorithmic}
\STATE{$G \leftarrow  g(X)$}
\STATE{$C \leftarrow  \exp(- G h \gamma)$}
\STATE{$Z \sim {\cal N}(0,1)$}
\STATE{$P\leftarrow  C P + \sqrt{\beta^{-1} (1-C^2)} Z$} 
\RETURN $P,\ G$
\end{algorithmic}
\end{algorithm}
If we compute the correction term in step O we have
\begin{align}
\tilde{\rm{B}}&:= 
\begin{cases}
dx&=0\\
dp &= -\nabla V(x(t)) g(x(t)) dt
\end{cases}\\
\tilde{\rm{O}}&:=
\begin{cases}
   d x &=0\\
 d p &=-\gamma p g(x)dt +\nabla g(x) \beta^{-1} dt + \sqrt{2\beta^{-1} \gamma g(x)}dW(t).
    \end{cases}
\end{align}
Similarly, the step B step can be resolved using the Euler method while the steps O may be solved exactly as well, as these are OU process and the correction term only depends on $x$, which is constant. For an OU process $dx=\theta(\mu-x_t)dt+\sigma dW(t)$, the drift terms are $\theta = \gamma g(x)$ and $\mu = \frac{\nabla g(x) \beta^{-1}}{\gamma g(x)}$. The solution has mean $p_0 \exp(-\gamma g(x) dt) + \frac{\nabla g(x)\beta^{-1}}{\gamma g(x)} (1-\exp(-\gamma g(x) dt)$ and variance $\sigma^2=\frac{1}{\beta} (1-\exp(-2 \gamma g(x) dt))$. This leads to Algorithms \ref{alg:stepB} and \ref{alg:stepO_hat}. 
\begin{algorithm}[H]
\caption{A step of size $h$ for $\tilde{\rm{B}}$.  Inputs: $X,\ P,\ h,\ F,\ G,\ G_p$.} \label{alg:stepB}
\begin{algorithmic}
\STATE{$F= -\nabla V(X)$}
\STATE{$G = g(X)$}
\STATE{$G_p = \nabla g(X)$}
\STATE{$P:=P+h G F$} 
\RETURN $P,\ F,\ G,\ G_p$
\end{algorithmic}
\end{algorithm} 
\begin{algorithm}[H]
\caption{A step of size $h$ for $\tilde{\rm{O}}$.  Inputs: $X,\ P,\ h,\ G$.} \label{alg:stepO_hat}
\begin{algorithmic}
\STATE{$G \leftarrow  g(X)$}
\STATE{$C \leftarrow  \exp(- G h \gamma)$}
\STATE{$Z \sim {\cal N}(0,1)$}
\STATE{$P:= C P +\frac{G_p (1-C)}{\beta \gamma G}+\sqrt{\beta^{-1} (1-C^2)} Z$} 
\RETURN $P,\ G$
\end{algorithmic}
\end{algorithm} 

For step A, the integration is a little more complicated.  Due to the introduction of the monitor function $g$, the vector field becomes
\begin{align}
\hat{\rm{A}}:=
\begin{cases}
   d x&=  p g(x) dt\\
d p  &=0
\end{cases}
\end{align}
To obtain second order schemes, we need to discretize this ODE system by an appropriate method.  We elect to use the implicit midpoint method and to solve the implicit equations by fixed point iteration.  We write this as a step of length $h$, as follows:
\begin{align} \label{eq:stepA}
    {x}_{n+1}={x}_n+ h {p}_n g\left(\tfrac{1}{2}(x_{n+1}+x_{n})\right),
\end{align}
where we iterate over the value of $x^j_{n+1}$. The iterations yield 
\begin{align}
    {x}^{j}_{n+1}={x}_n+ h {p}_n g\left(\tfrac12(x^{j-1}_{n+1}+x_{n})\right),
\end{align}
where the first guess is given by $x^0_{n+1}={x}_n+ h {p}_n g(x_{n})$ and $j=1, \dots, J$.  We set a tolerance $10^{-12}$ and a maximum number of iterations $n_{\max}=100$. In Figure \ref{fig:stepA_exploration}, the maximum number of iterations required does not explode and the required tolerance is reached on average in under 6 steps. The difference in the last two iterations of the fixed point algorithm is always below $10^{-12}$, as required by the algorithm. Note that for the accuracy results, we set the tolerance to ${10}^{-16}$. 
\begin{figure}[htp]
\begin{center}
\includegraphics[width=3in]{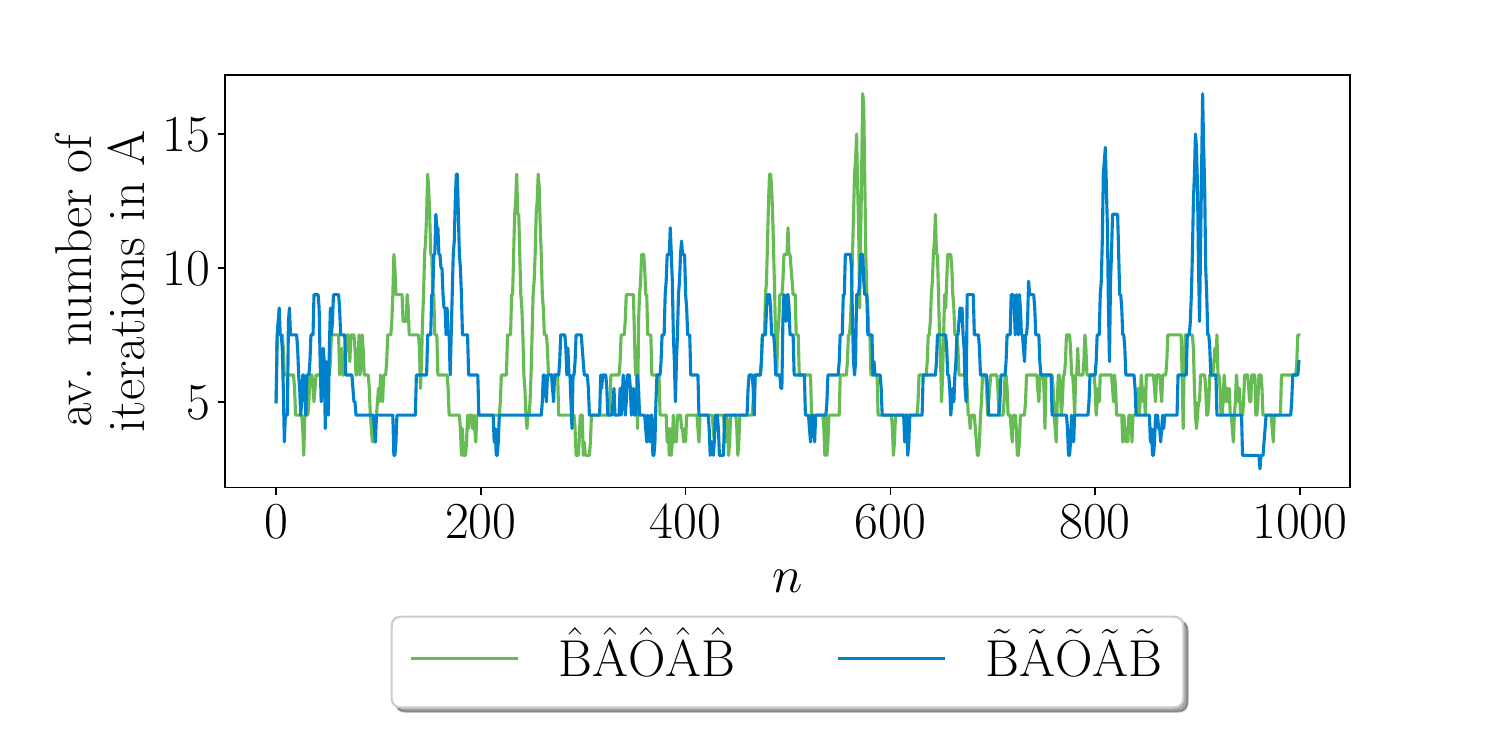}
\caption{The average number of iterations over the two steps A to reach the required tolerance for the algorithms $\hat{\rm{B}}\hat{\rm{A}}\hat{\rm{O}}\hat{\rm{A}}\hat{\rm{B}}$ and $\tilde{\rm{B}}\tilde{\rm{A}}\tilde{\rm{O}}\tilde{\rm{A}}\tilde{\rm{B}}$.}
\end{center}    
\label{fig:stepA_exploration}
\end{figure}

In the sequel we denote the various algorithms by $\hat{A}(h)$,$\tilde{A}(h)$, $\hat{B}(h)$,$\tilde{B}(h)$, $\hat{O}(h)$, $\tilde{O}(h)$, where, among the inputs, just the stepsize is explicitly indicated.
We thus obtain two schemes by composition:
\[
\hat{\Psi}_{h} =\hat{B}(h/2)\hat{A}(h/2)\hat{O}(h)\hat{A}(h/2)\hat{B}(h/2)
\]
\[
\tilde{\Psi}_{h} = \tilde{B}(h/2)\tilde{A}(h/2) \tilde{O}(h)\tilde{A}(h/2)\tilde{B}(h/2).
\]
Each method requires at least four evaluations of the monitor function $g$ and two of the function $g'$, and these evaluations are potentially more computationally intensive for step A, depending on the number of iterations needed. We make the working assumption that the function $g: \mathbb{R^d} \to \mathbb{R}$ can be constructed based on a greatly simplified expression compared to the full computation of the potential energy function, so we anticipate that the added computational cost is relatively small in large dimensional settings. We note that the steps A, B and O may also be used within adaptive versions of alternative splitting schemes such as OBABO, ABOBA and the stochastic Verlet position (SPV) method \cite{Melchionna2007}, where the first and last half step are computed by a similar approach to the fixed point integration required for A, while the middle step can be computed exactly as step $\rm{\tilde{O}}$. The Python and C\texttt{++} source code as well as the environment for reproducing our results are available online.\footnote{Source Code: \href{https://github.com/alixsleroy/Adaptive-stepsize-algorithms-for-Langevin-dynamics}{https://github.com/alixleroy/Adaptive-stepsize-algorithms-for-Langevin-dynamics}}

\section{Numerical experiments}
\label{sec:numerical_experiments}
After reviewing the properties of the discretization schemes applied to the the IP-transformed SDEs \eqref{eq:transformed_overdamped_sde} and \eqref{eq:th_transformed_underdamped}, we describe numerical simulations to illustrate the results. The examples illustrate that computational effort can be reduced while obtaining a similar accuracy. We compute functionals of the processes to evaluate the weak error. The error in the sampled distribution $\hat{\rho}$ is reduced to an average of an observable $\phi$ at time $t>0$ \cite{LeimMatTret2014}, the true value $\bar{\phi}$ is computed analytically or by quadrature via \emph{scipy.integrate.quad} \cite{2020SciPy-NMeth}. The observable estimated by a numerical timestepping scheme with step $h$ is $\hat{\phi}(t,h)$. We have
\begin{align*}
\bar{\phi}(t)&=E_{\rho(t,.)} \phi = \int_{\mathbb{R}^d} \phi(x) \rho(t,x) dx, \\
    \hat{\phi}(t,h)&=E_{\rho_N(.)} \phi =\frac{1}{M} \sum_{j=1}^{M} \phi(X_N^{j}),
\end{align*}
where the sample $X_N^{j} \approx x(N h)$ and we posit the use of a large number of trajectories $M$ to reduce the sampling error. A time-discrete approximation converges weakly with order $r > 0$ at time $T_f$ if there exists a positive constant $C$ which does not depend on the stepsize $h$ such that $ |\hat{\phi}(t)-\bar{\phi}(t,h)|\leqslant C h^r$ \cite[chapter~9, section~6]{kloedenPlaten1999}. A good observable ideally captures most of the behaviour of the distribution of interest. We choose the observables as the moments up to order $k \in \mathbb{N}_0$ such that $\phi(x) = x^k$ as in  \cite[Chapter~9, Section~4]{kloedenPlaten1999}. 
\subsection{Modified harmonic potential}
\subsubsection{Overdamped transformed dynamic}
\label{sec:over_acc}
Recall the example presented in Section \ref{sec:running examples}, with the potential given by \eqref{eq:pot_spring}, plotted in Figure \ref{fig:spring_left}, and the choice of monitor function given by \eqref{eq:monitor_phi} with $g_3(x)=\psi(\omega(x))$. 
In Figure \ref{fig:accuracy_spring_steep} we observe an order of convergence that is similar for the Euler-Maruyama method applied to the transformed and untransformed Langevin equations. The asymptotic regime is reached for a larger stepsize for the IP-transformed overdamped dynamics. This gain in efficiency for the transformed dynamics is not due to a smaller average step, as the average value taken by the monitor function is $1.18$. It implies that a larger average step yields better results. This shows a gain in efficiency. While Figure~\ref{fig:tr_large_h} illustrates that the discretization \eqref{eq:em_tr_sde} is more stable numerically, the results in Figure~\ref{fig:accuracy_spring_steep} show that the discretized process \eqref{eq:em_tr_sde} yields better samples for a lower computational budget. In other words, it approaches the steady state faster than the discretized dynamics \eqref{eq:em_sde}. These results highlight a gain in stability and efficiency.

\begin{figure}[htp]
\hfill
\subfigure[\footnotesize{Final time is set to $T_f=100$ and stepsize between $0.3$ and $1.48$ with a mean of the monitor function $1.2$, number of trajectories is $n_s=10^8$, $M=1.5$, $m=0.001$, $\beta^{-1}=0.1$ ($a=2.75$, $b=0.1$, $c=0.1$, $x_0=0.5$). The monitor function takes $r=1$ and $\alpha=2$.}]{\includegraphics[width=6.cm]{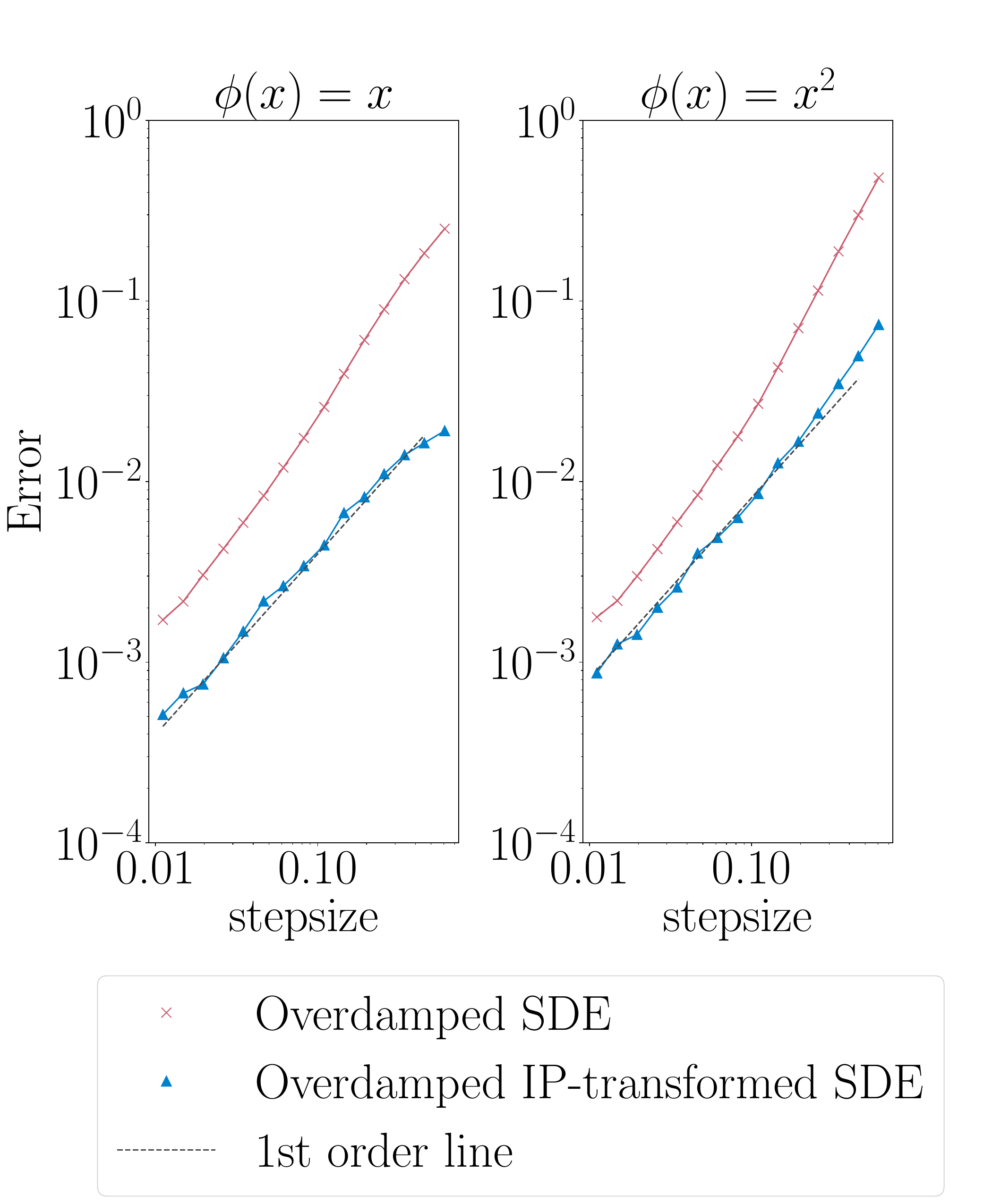}\label{fig:accuracy_spring_steep}}
\hfill
\subfigure[\footnotesize{The modified harmonic potential is closer to a simple harmonic potential as we set the parameters to $a=1$ and $b=1$. The stepsizes are set between $0.6$ and $1.37$ with a mean value of the monitor function of $0.77$. The rest of parameters are similar to Figure \ref{fig:accuracy_spring_steep}.}]{\includegraphics[width=6.cm]{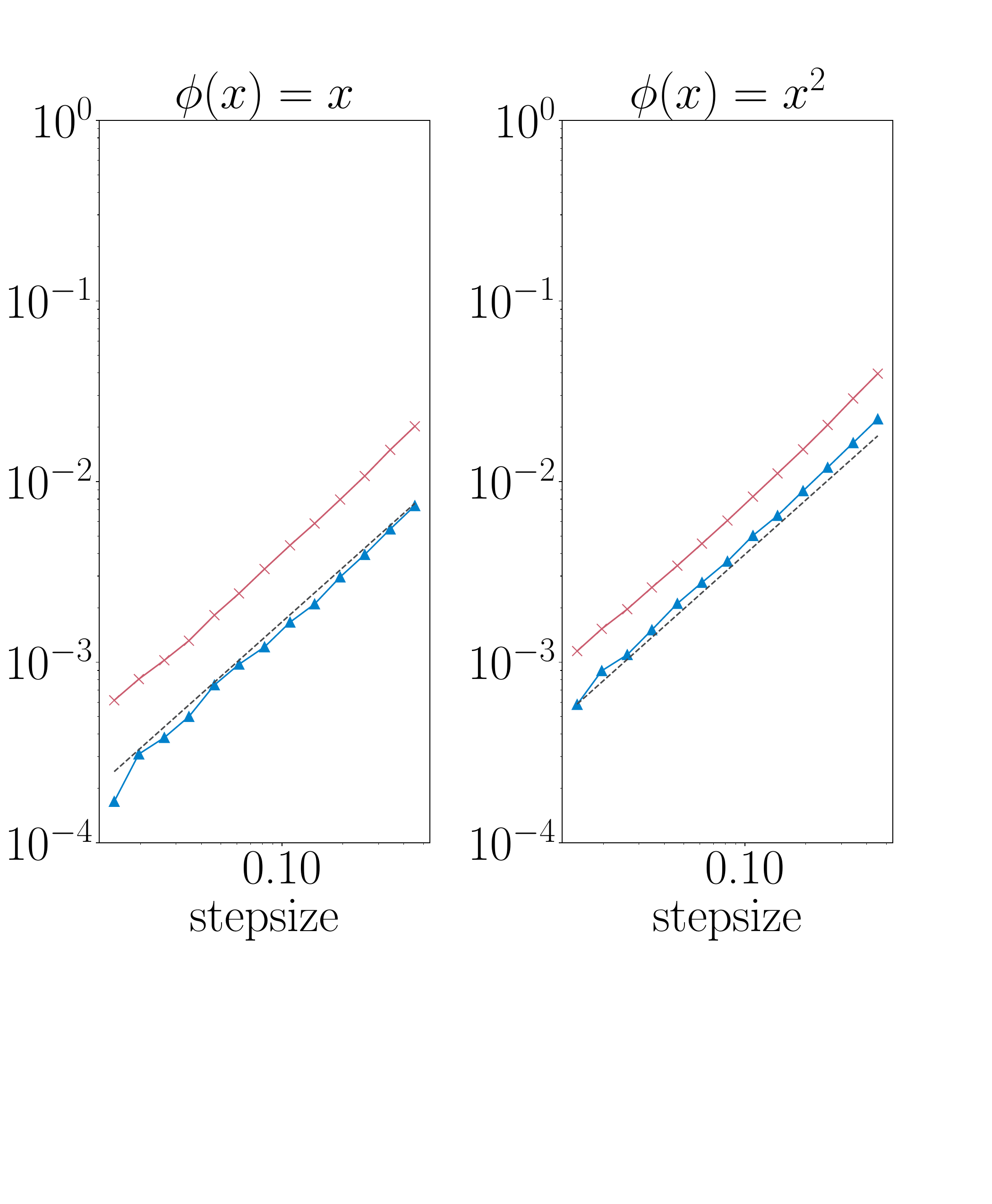}\label{fig:accuracy_spring_easy}}
\hfill
\caption{Order of accuracy in the case of a modified harmonic potential.}
\end{figure}

To demonstrate that there are no changes in the order of convergence of the method, a similar simulation was performed with parameters that render the potential less steep, and easier to integrate. In Figure \ref{fig:accuracy_spring_easy}, the order of convergence of the Euler-Maruyama method is recovered for larger stepsize for both dynamics. 
The potential is close to harmonic. The smaller error arising in the transformed case is this time due simply to the smaller average value taken by the monitor function, $0.77$. We note that this example is easy to numerically integrate and no benefits can be gained from the transformed dynamics. 
\subsubsection{Underdamped transformed dynamic}
In this section, we provide results comparing the BAOAB scheme applied to the system \eqref{eq:full_langevin_1d} and the discretization methods developed in Section \ref{sec:numerical_cons}, $\hat{\rm{B}}\hat{\rm{A}}\hat{\rm{O}}\hat{\rm{A}}\hat{\rm{B}}$ and $\tilde{\rm{B}}\tilde{\rm{A}}\tilde{\rm{O}}\tilde{\rm{A}}\tilde{\rm{B}}$, with similar choices of potential and monitor function to those in Section \ref{sec:over_acc}. 
\begin{figure}[htp]
\centering
    \includegraphics[width=.9\textwidth]{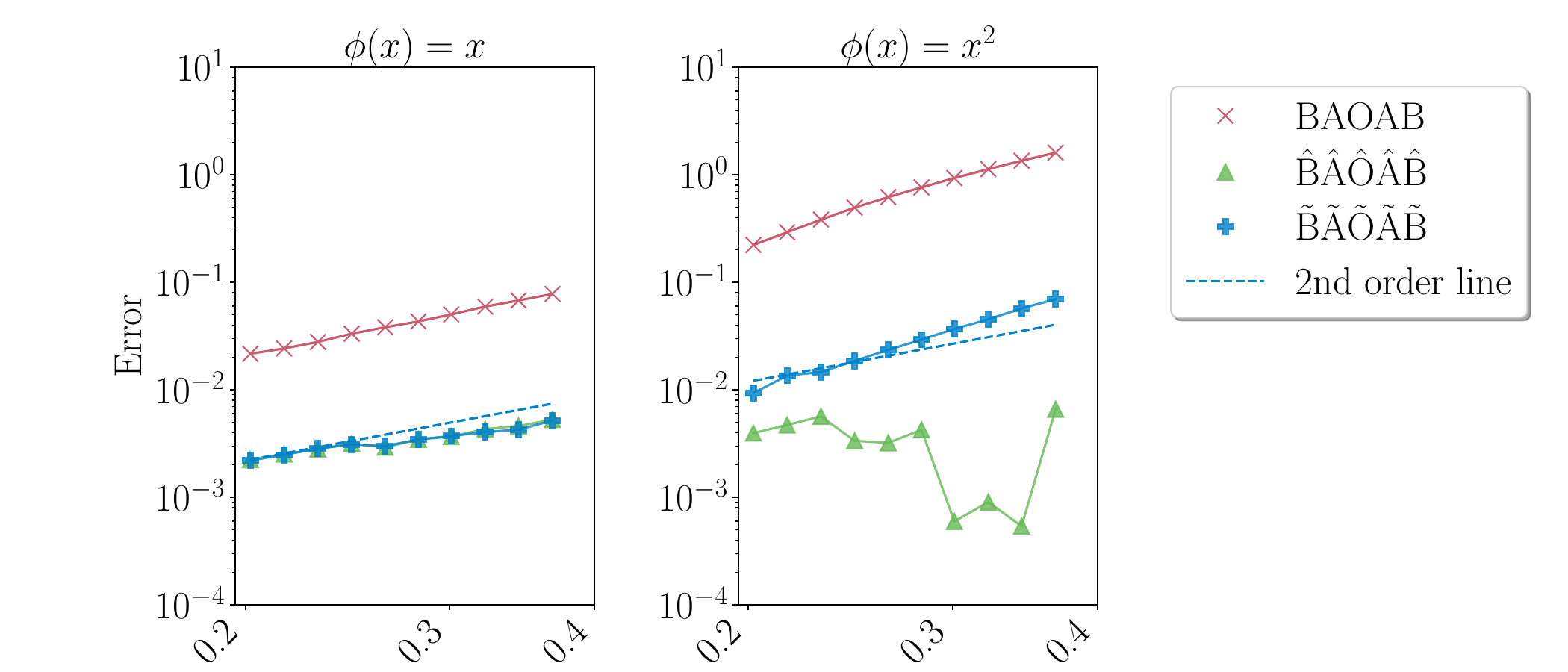}
     \caption{ \footnotesize{Observable in the case of a modified harmonic potential with final time $T_f=40000$, number of burn-in steps $10000$, tolerance for step A set to $10^{-11}$ and maximum number of iterations is $100$. The number of trajectories is $n_s=10^5$, and bounds on the values of the monitor functions are $M=1.1$, $m=0.1$, with parameters $\beta^{-1}=1.$ and $\gamma=0.1$ ($a=2.75$, $b=0.1$, $c=0.1$, $x_0=0.5$). The monitor function takes $r=1$ and $\alpha=2$.}}
    \label{fig:accuracy_spring_easy_under}
\end{figure}
In Figure \ref{fig:accuracy_spring_easy_under}, we observe that the schemes $\hat{\rm{B}}\hat{\rm{A}}\hat{\rm{O}}\hat{\rm{A}}\hat{\rm{B}}$ and $\tilde{\rm{B}}\tilde{\rm{A}}\tilde{\rm{O}}\tilde{\rm{A}}\tilde{\rm{B}}$ have much higher efficiency and display a quadratic decay in the error, similarly to the original schemes \cite{LeimkuhlerMatthews2013}. We note that for these simulations, the average value taken by the monitor function is between $0.991$ and $0.993$. 
As mentioned in Section~\ref{sec:numerical_cons},
we can implement different splitting methods in a similar way. In Figure \ref{fig:accuracy_spring_mult_under} we show accuracy results for a range of splitting schemes for the IP-transformed system \eqref{eq:th_transformed_underdamped}.
\begin{figure}[htp]
\centering
\includegraphics[width=.7\textwidth]{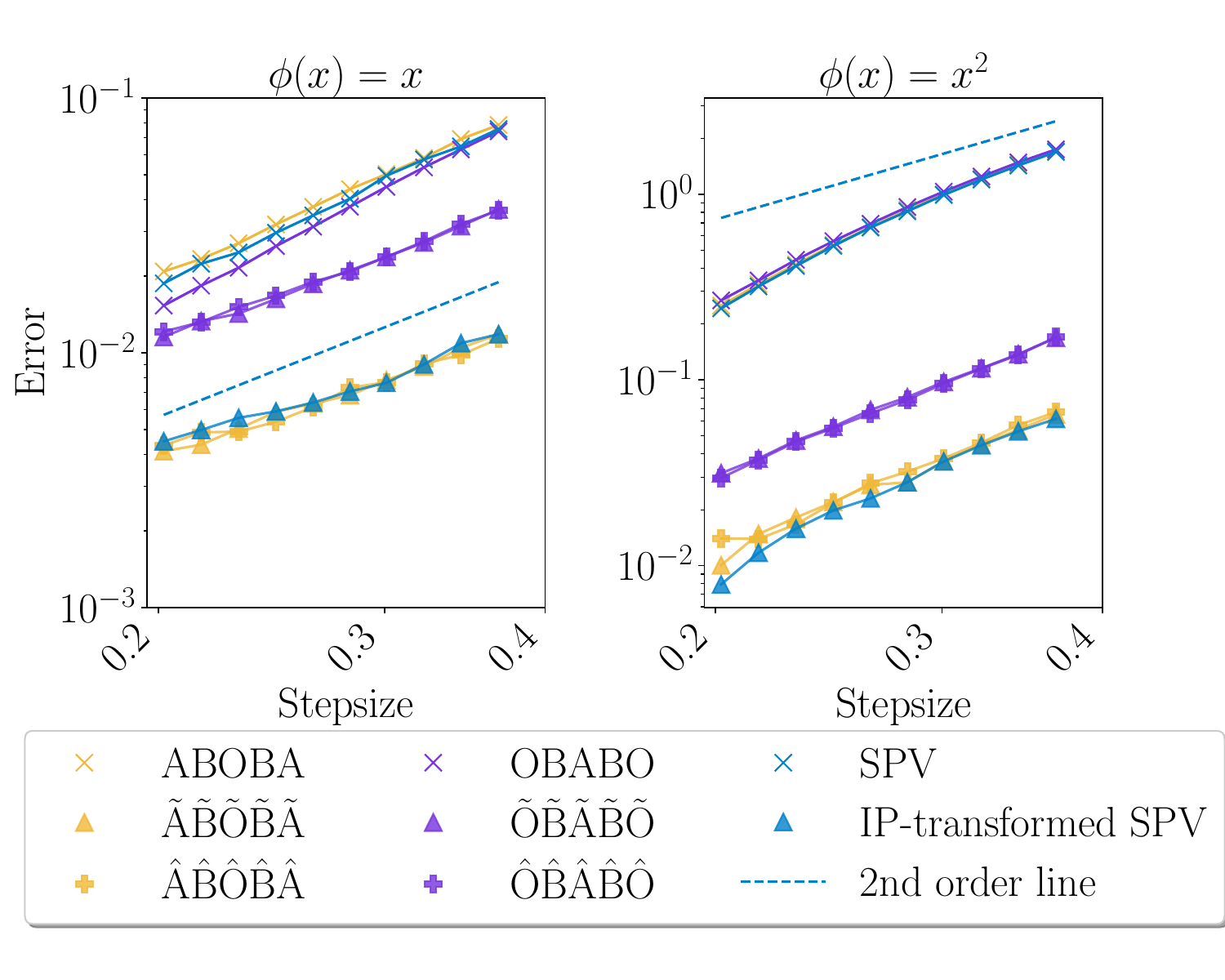}
     \caption{ \footnotesize{Observables in the case of a modified harmonic potential with similar setting as Figure \ref{fig:accuracy_spring_easy_under} comparing different splitting schemes.}}
    \label{fig:accuracy_spring_mult_under}
\end{figure}
We see that the algorithms $\hat{\rm{A}}\hat{\rm{B}}\hat{\rm{O}}\hat{\rm{B}}\hat{\rm{A}}$, $\tilde{\rm{A}}\tilde{\rm{B}}\tilde{\rm{O}}\tilde{\rm{B}}\tilde{\rm{A}}$, and the IP-transformed SPV perform better than the $\hat{\rm{O}}\hat{\rm{B}}\hat{\rm{A}}\hat{\rm{B}}\hat{\rm{O}}$ and $\tilde{\rm{O}}\tilde{\rm{B}}\tilde{\rm{A}}\tilde{\rm{B}}\tilde{\rm{O}}$ schemes. 
We leave for future work the rigorous analysis required to 
explain the relative performance of the splitting schemes. Both $\hat{\rm{B}}\hat{\rm{A}}\hat{\rm{O}}\hat{\rm{A}}\hat{\rm{B}}$ and $\tilde{\rm{B}}\tilde{\rm{A}}\tilde{\rm{O}}\tilde{\rm{A}}\tilde{\rm{B}}$ schemes are recommended for their superior performance. We also suggest to the reader to implement the IP-transformed SPV over the schemes $\hat{\rm{A}}\hat{\rm{B}}\hat{\rm{O}}\hat{\rm{B}}\hat{\rm{A}}$, $\tilde{\rm{A}}\tilde{\rm{B}}\tilde{\rm{O}}\tilde{\rm{B}}\tilde{\rm{A}}$ for ease of implementation. We note that convergence analysis of these schemes is necessary before encouraging wider adoption of those methods.
In summary, the main advantage of these adaptive methods resides in the fact that the asymptotic behaviour is reached for larger stepsizes. Thus, they provide a useful approximation of the distribution more efficiently. Note that we have changed the temperature (by increasing $\beta^{-1}$) in the underdamped simulations compared with the overdamped simulations presented in Figure~\ref{fig:accuracy_spring_easy} -- we did this mainly to simplify the example.

\subsection{A Bayesian example}
\label{sec:bayesian}
Bayesian inference \cite{Gamerman2006suitable} often requires sampling algorithms to approximate so-called posterior distributions.  When the Bayesian framework is used to represent variables on bounded sets, e.g., with steep priors \cite{UribeDongHansen2023}, Langevin-based sampling algorithms are difficult to apply. Hard bounds need to be encoded with boundary conditions, softer bounds produce large gradients in the potential close to the boundary.
We now consider such an inference problem, where we  observe data $y_1,\ldots,y_N \sim \mathrm{N}(\mu,1)$ i.i.d. with an unknown $\mu \in \mathbb{R}$. A priori, we know that the true value of $\mu$ lies in the interval $[1,3]$, which we enforce through a prior distribution that is smoothly bounded. We aim to sample from the posterior distribution $\rho_{\rm Bayes}$ in this setting given by
\begin{align*}
    \rho_{\rm Bayes}(\mu) \propto \prod_{i=1}^N \exp \left(- \tfrac{1}2\|y_i-\mu\|^2 \right)\exp \left( -\left(\mu-a\right)^{2K}\right),
\end{align*}
where the parameter $K$ controls the steepness of the bounds in the approximately uniform prior. As before, we aim at sampling from the potential $V(\mu) = \log p(\mu|\bold{x}) \pi(\mu)$ such that
\begin{align}
V' (\mu)&= -\left({\sum}_{i=1}^{N} y_i - N \mu - 2 K(\mu-a)^{2 K -1}\right),
\end{align}
which can be visualiced in Figure \ref{fig:bayesian_right}. An adaptive function adjusted to the prior and the likelihood is
\begin{align}
g(\mu) = \psi\left(2(\mu-a)+\left(\frac{1}N\sum^N_{i=1} y_i-a\right)^2\right),
\end{align}
as seen in Figure \ref{fig:bayesian_g}. The parameters used are $K=4$, $a=2$, $\beta^{-1}=1$, $\gamma=0.1$ and $m=0.1$ and $M=1$. The data are simulated using $y_i \sim \mathcal{N}(1.7,1)$ with $N=10$. The average step in the BAOAB algorithm is adjusted to the average value of the monitor function in the $\tilde{\rm{B}}\tilde{\rm{A}}\tilde{\rm{O}}\tilde{\rm{A}}\tilde{\rm{B}}$ algorithm.

\begin{figure}[htp]
\hfill
\subfigure[The gradient $V'$.]{\includegraphics[width=4cm]{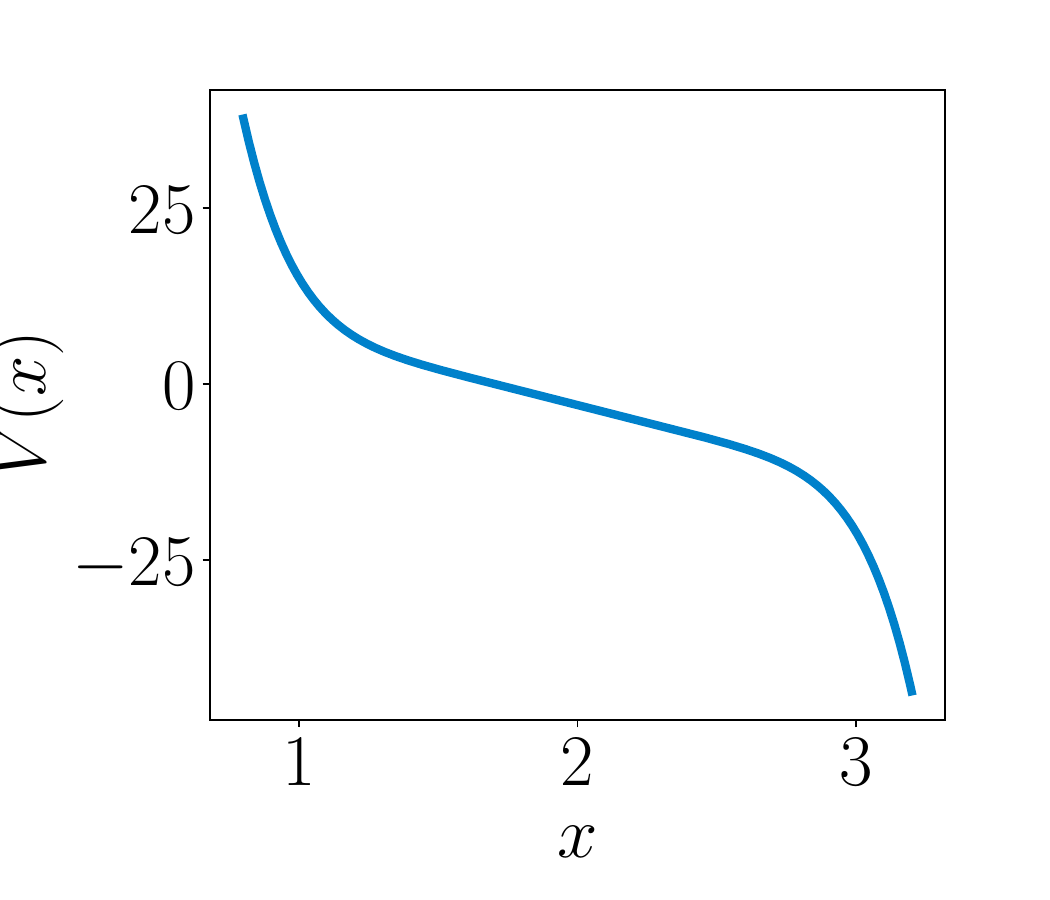}\label{fig:bayesian_right}}
\hfill
\subfigure[The monitor function $g$ with $\alpha=2$ and $r=2$.]{\includegraphics[width=4cm]{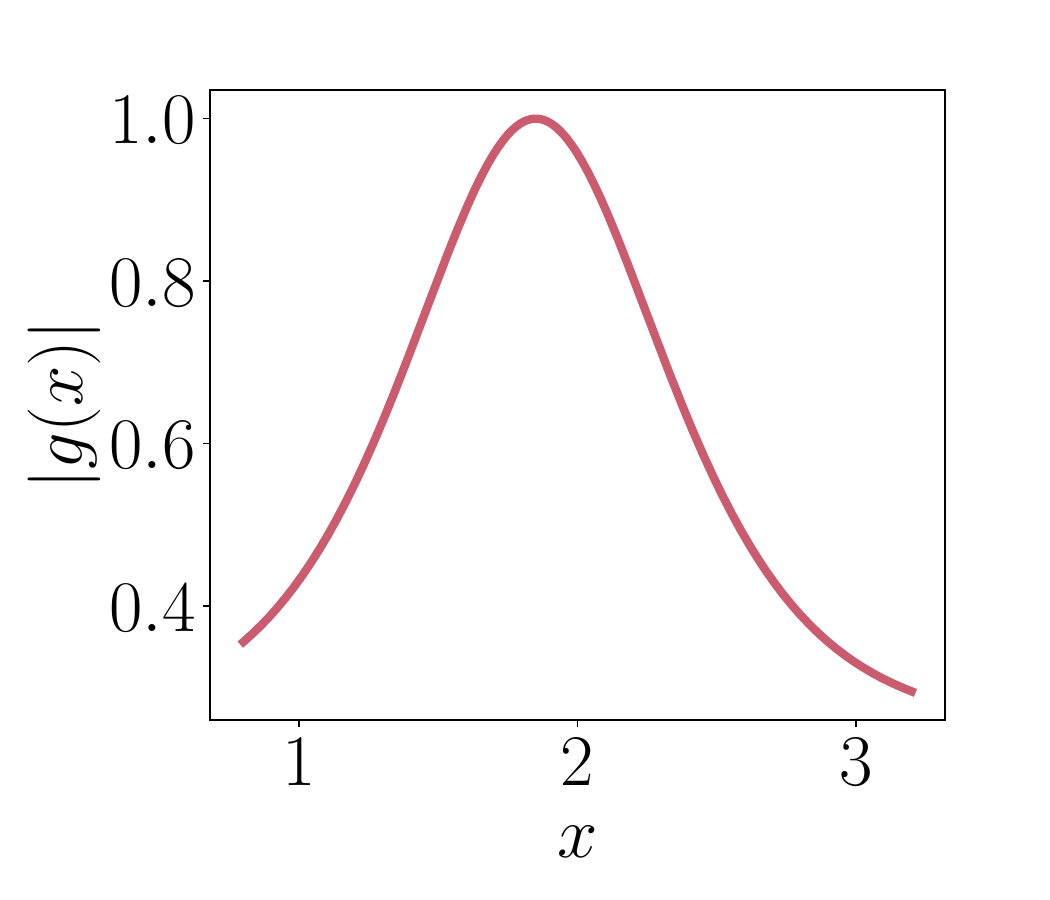}\label{fig:bayesian_g}}
\hfill
\subfigure[The percentage of escaping trajectories over $10^5$ runs with the three different algorithms until $T_f=10^4$.]{\includegraphics[width=4cm]{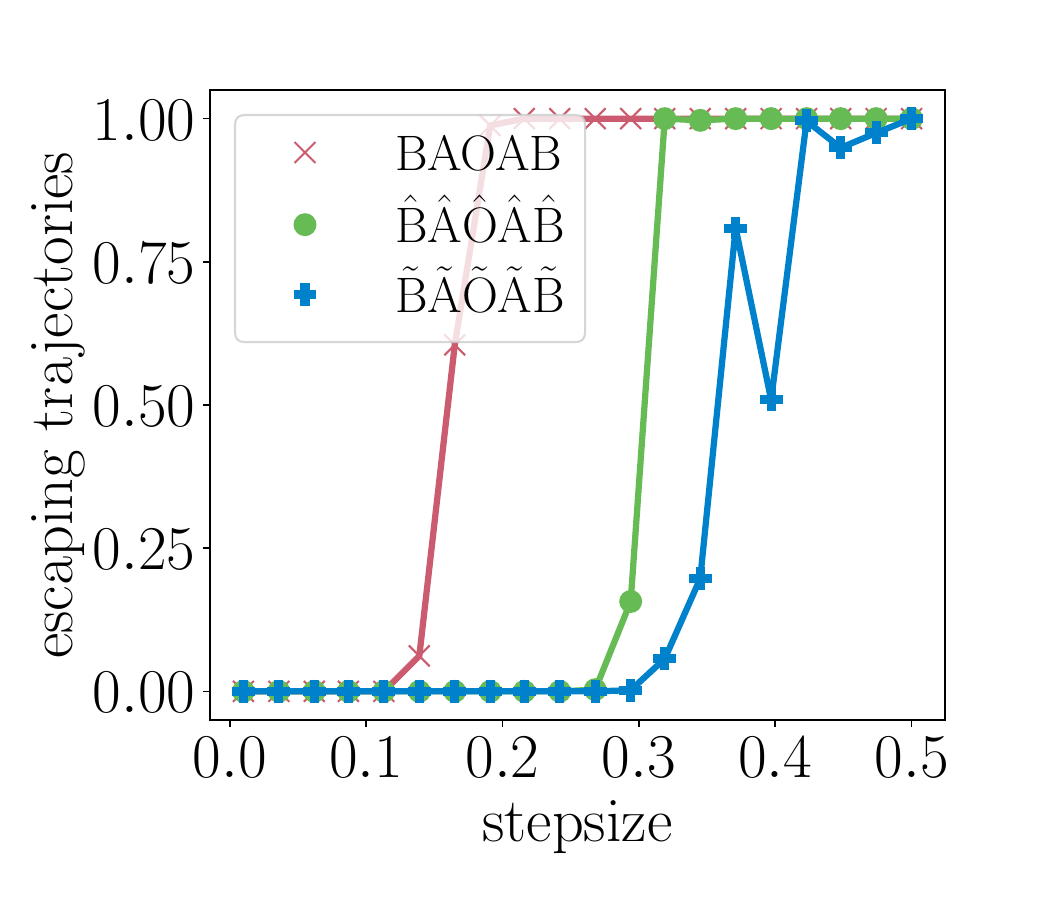}\label{fig:bayesian_escaping}}
\hfill
    \caption{Plots for the Bayesian example and proportion of trajectories that are unstable for different stepsizes.}
\end{figure}

In Fig \ref{fig:bayesian_escaping}, we see that the adaptive algorithm yields fewer escaping trajectories for a similar stepsize. The stepsize used in $\rm{B}\rm{A}\rm{O}\rm{A}\rm{B}$ is rescaled by the smallest average value taken by the monitor function $g$ over the simulations in $\tilde{\rm{B}}\tilde{\rm{A}}\tilde{\rm{O}}\tilde{\rm{A}}\tilde{\rm{B}}$, which is $0.85$. This means that the adaptive algorithm is more stable in the sense of generating fewer unstable trajectories. The adaptive methods 
therefore allow larger stepsizes, saving computational effort.

\subsection{Example 2D: a system with two pathways}
We next turn to a more complicated example. This system is still planar but we now introduce two channels: one narrow valley in the energy landscape and a second, set at a slightly higher energy level, which is much wider and shallower, and is thus entropically favored at modest temperature.  The energy function is
\begin{equation}
q(x,y) = \frac{1+k_1 p_1(x,y) p_2(x,y)}{1+p_1(x,y)}+ k_3 \frac{k_2 p_1(x,y) p_2(x,y)}{1+k_2 p_2(x,y)} + k_4 x^2, 
\end{equation}
where $p_1(x,y) = (y-x^2+4)^2$ and $p_2(x,y) = (y+x^2-4)^2$. In our experiments we set $k_1 = 0.1,\ k_2 = 50,\ k_3=50,\ k_4=0.1$. The parameters can be tuned to influence the difficulty of the numerical integration of the potential. 
The dynamics of this system unfold in the vicinity of two parabolic arcs, with the contour plot of the potential shown in Figure \ref{fig:two_paths_2dplot}.  In 
Figure \ref{fig:two_paths_traj} (b) 
we used using a fixed-stepsize Langevin integrator with a small stepsize. 
\begin{figure}[htp]
\hfill
\subfigure[The dynamics will be confined to the narrow upper and wider lower channels.]{\includegraphics[width=4.5cm]{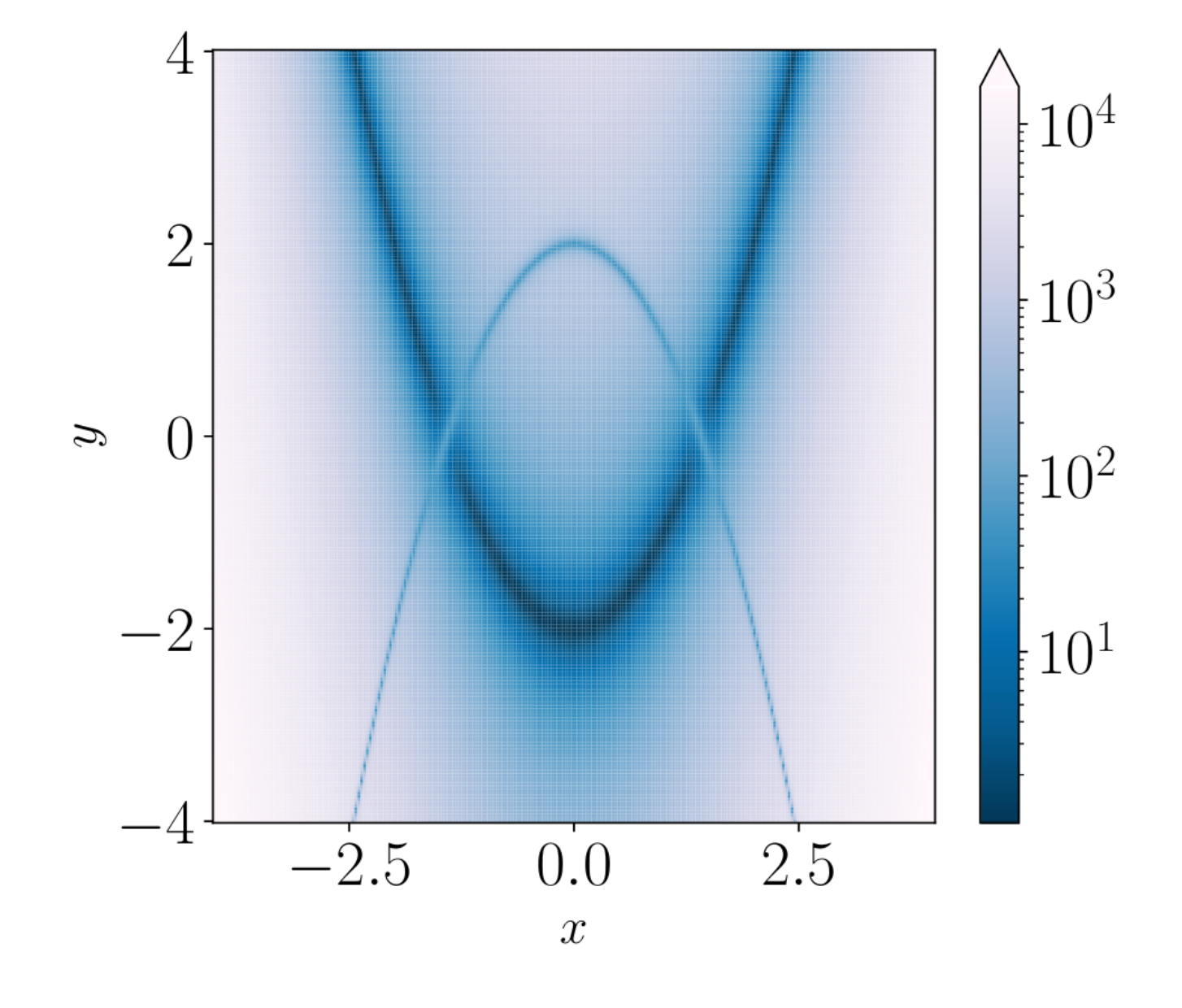}}\label{fig:two_paths_2dplot}
\hfill
\subfigure[Fixed stepsize Langevin sampling using a small stepsize $h=0.005$, $T_f=100000$, a reasonable approximation to the correct distribution. Trajectories are plotted with $\beta^{-1}=0.1$ and $\gamma=0.5$ every 10 steps.]
{\includegraphics[width=4.5cm]{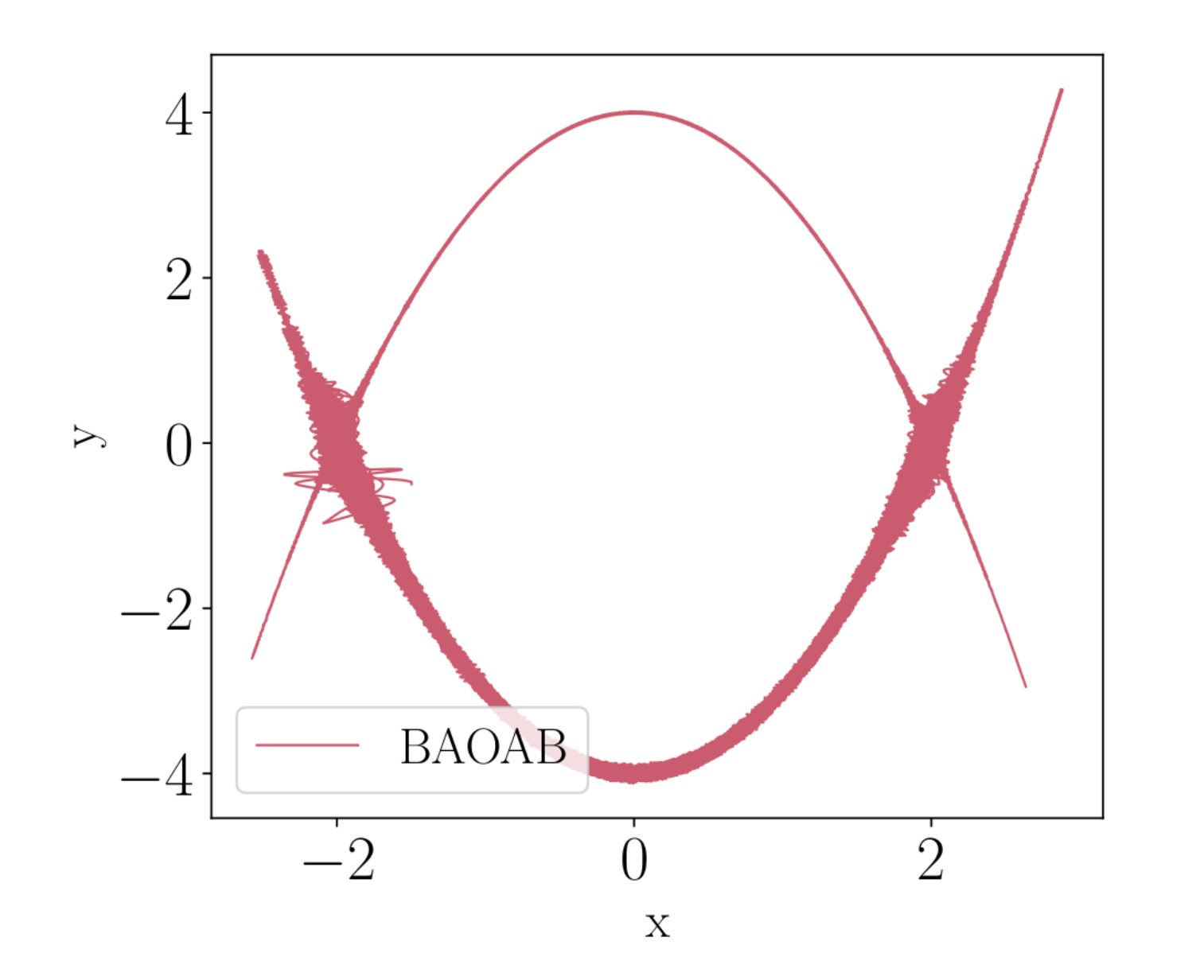}} \label{fig:two_paths_traj} 
\hfill
\caption{The two-path problem with both wide and narrow valleys.} \label{fig:two_paths}
\end{figure}
In Figure~\ref{fig:two_paths_vs} (a),  with a zoomed version in 
Figure~\ref{fig:two_paths_vs} (b), 
this simulation is reproduced using a larger stepsize ($h=0.025$) alongside the results of the adaptive algorithm $\tilde{\rm{B}}\tilde{\rm{A}}\tilde{\rm{O}}\tilde{\rm{A}}\tilde{\rm{B}}$. 
We see that the fixed-stepsize sampling of the target distribution is poor.
We suggest that the larger stepsize prevents trajectories from entering the narrow channel. 
For the adaptive algorithm, we designed a monitor function that reduces the stepsize when close to the upper narrow channel. We us the format presented in \eqref{eq:monitor_phi} with 
$
    \tilde{g}(x,y): = g(f(x,y))
    $ and 
    $
        f(x,y) = (y+x^2-4)^2
        $,
$m=0.2$ and $M=1$. In the results presented in Figure \ref{fig:two_paths_vs}, we adjusted the size of the step in the $\tilde{\rm{B}}\tilde{\rm{A}}\tilde{\rm{O}}\tilde{\rm{A}}\tilde{\rm{B}}$ method to match the average value taken by the monitor function $g$ over the entire trajectory; in that way we provide a comparison for a similar computational effort. The average value taken by the monitor function was $0.65$. 
\begin{figure}[htp]
\centering
\includegraphics[width=5cm]{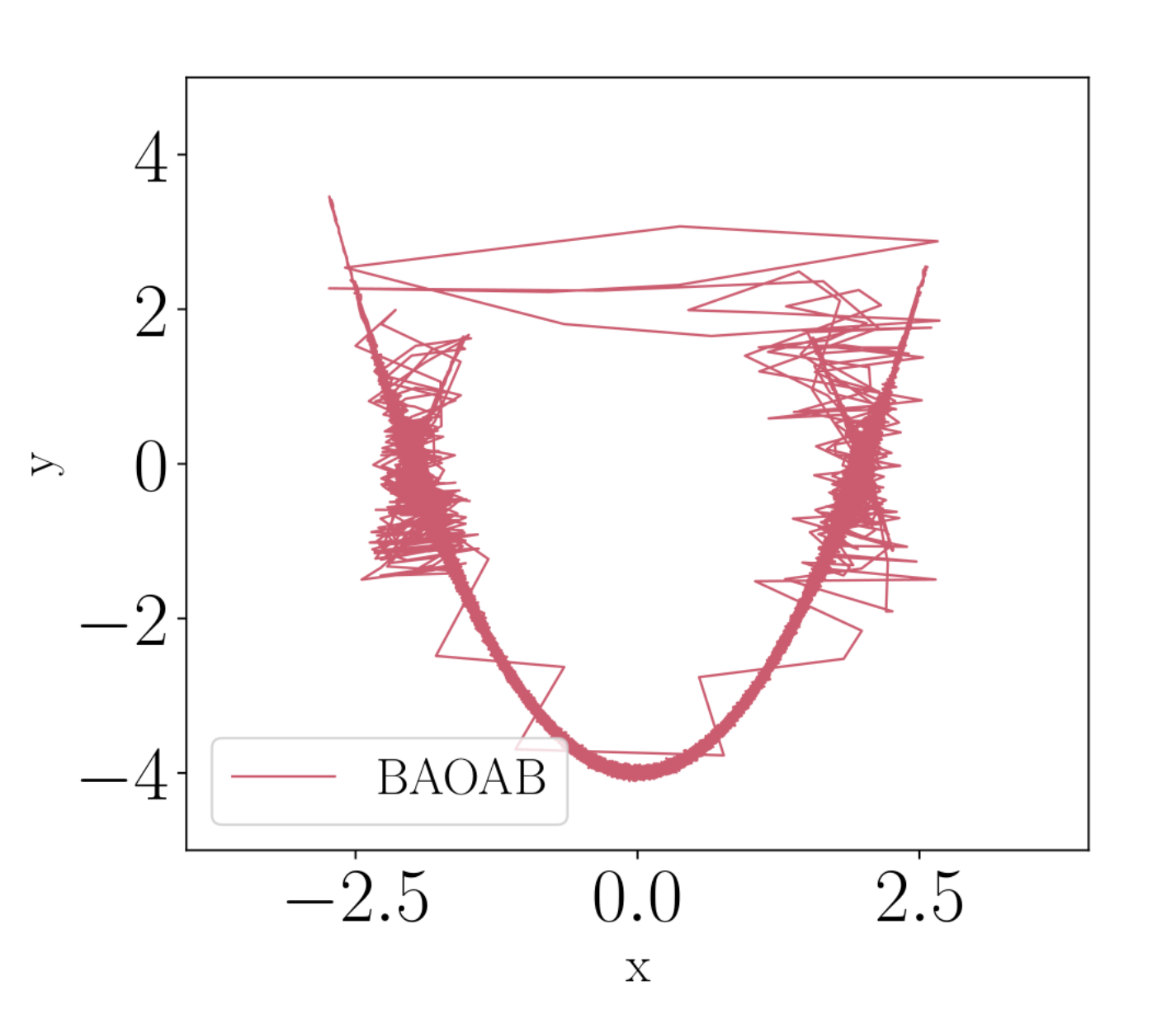}
\includegraphics[width=5cm]{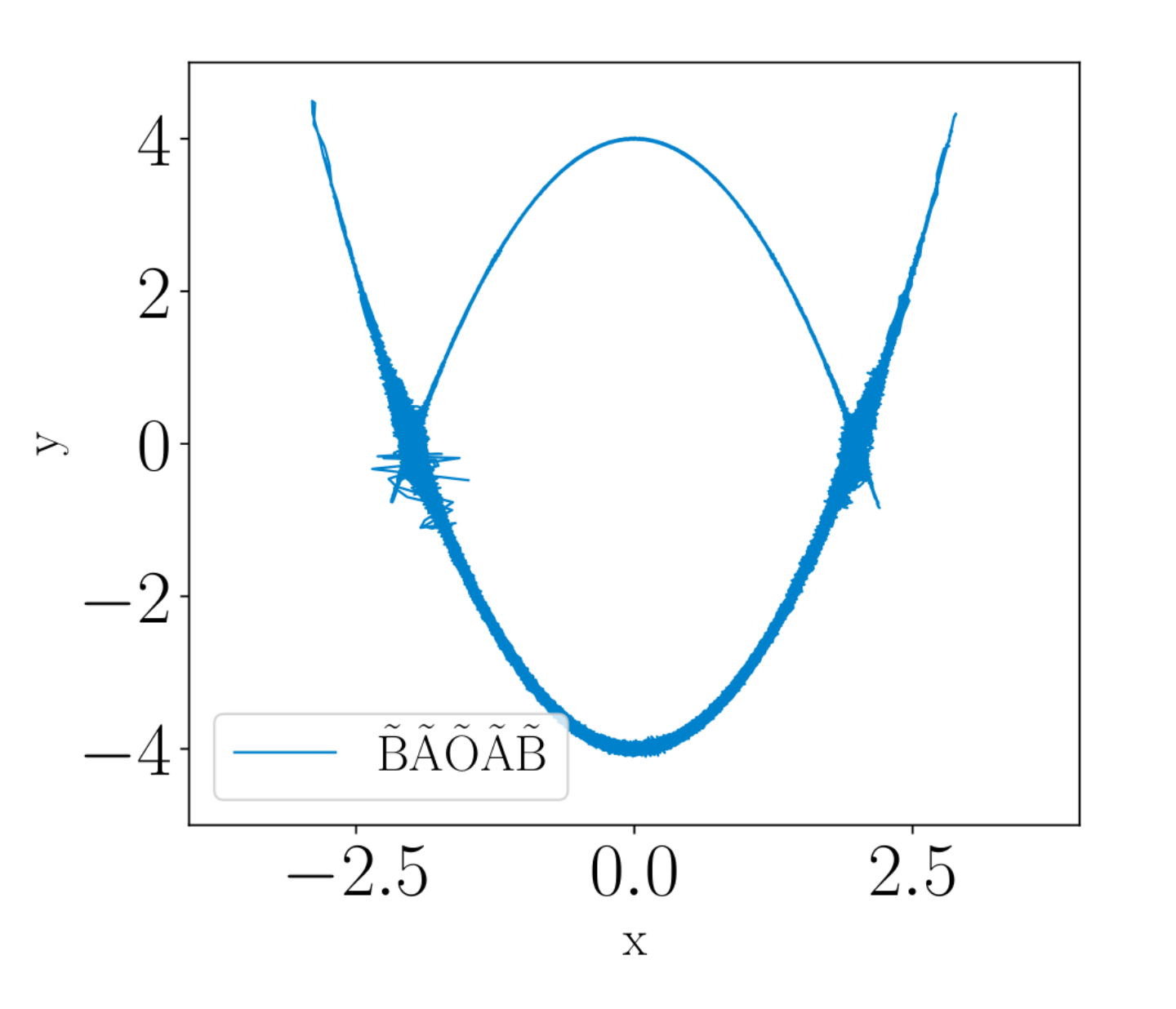} 
\caption{Plots of the trajectories with $\rm{B}\rm{A}\rm{O}\rm{A}\rm{B}$ with a stepsize rescaled by the average value taken by the monitor function ($0.73$) in $\tilde{\rm{B}}\tilde{\rm{A}}\tilde{\rm{O}}\tilde{\rm{A}}\tilde{\rm{B}}$ where $h=.0275$. The monitor function has $r=1$ and $\alpha=2$. The transformed system visits the upper narrow channel.}
\label{fig:two_paths_vs}
\end{figure}


\section{Conclusions}
\label{sec:discussions_conclusions}In this work, we explored stochastic differential equations that allow efficient approximation of
a target probability distribution. These invariant-preserving transformed dynamics are fundamentally based on an adaptive time stepping in overdamped and underdamped Langevin dynamics, respectively. The target measure is preserved by adding a correction term to the usual, direct time-rescaled dynamics, effectively changing the dynamics. Throughout this work, the adaptivity is controlled by a monitor function, and we provided both formal criteria and heuristics for its choice. The formal criteria ensure existence and uniqueness of weak solutions to the overdamped and underdamped invariant preserving transformed dynamics, as well as uniqueness of their invariant measure and ergodicity. The heuristics help users to design an appropriate monitor function that ensures efficiency. While numerical integration of the overdamped case is simple, we provided seven novel splitting schemes to integrate the invariant-preserving transformed underdamped system. Yet, a careful convergence analysis of these schemes is necessary for wider adoption. The efficiency of this new adaptive timestepping approach was illustrated  
on several examples.

\bibliographystyle{siamplain}
\bibliography{references}

\end{document}